\documentclass[12pt]{amsart}

\theoremstyle{definition}
\theoremstyle{plain}

\usepackage[left=2.8cm,top=2.8cm,right=2.8cm,bottom = 2.4cm]{geometry}
\usepackage{amsfonts}  
\usepackage{amsmath}
\usepackage{amscd}
\usepackage{amssymb} 
\usepackage{amsthm} 
\usepackage{bbm}
\usepackage{bm}
\usepackage{enumerate}
\usepackage{enumitem}
\usepackage{hyperref}
\usepackage{latexsym}
\usepackage{mathabx}
\usepackage{mathdots}
\usepackage{mathtools}
\usepackage{xfrac}
\usepackage[all,frame]{xy}
\usepackage{hyperref}

\allowdisplaybreaks

\DeclareMathOperator{\adj}{adj}
\DeclareMathOperator{\cof}{cofact}
\DeclareMathOperator{\Res}{Res}

\DeclareMathOperator{\Sp}{Sp}

\newcommand{\Ahat}{\widehat{A}}
\newcommand{\alphat}{\widetilde{\alpha}}

\newcommand{\bbar}{\overline{b} }

\newcommand{\betat}{\widetilde{\beta}}
\newcommand{\bfm}{\bm{m}}
\newcommand{\bfy}{\bm{y}}

\newcommand{\C}{\mathbb{C}}

\newcommand{\calC}{\mathcal{C}}
\newcommand{\calE}{\mathcal{E}}
\newcommand{\calF}{\mathcal{F}}

\newcommand{\calH}{\mathcal{H}}
\newcommand{\calK}{\mathcal{K}}
\newcommand{\calL}{\mathcal{L}}
\newcommand{\calM}{\mathcal{M}}
\newcommand{\calN}{\mathcal{N}}
\newcommand{\calD}{\mathcal{D}}

\newcommand{\calX}{\mathcal{X}}

\newcommand{\Cg}{\mathfrak{C}_{g}}

\newcommand{\Chat}{\widehat{\C}}

\newcommand{\D}{\mathcal{D}}
\newcommand{\del}{\partial}
\newcommand{\delx}{\nabla(x)}

\newcommand{\delMg}{\nabla_{\hspace{-1 mm}\Mg}}

\newcommand{\End}{\textup{End}}

\newcommand{\half}{\frac{1}{2}}

\newcommand{\I}{\mathcal{I}}
\newcommand{\Id}{\textup{Id}}
\newcommand{\Ig}{\textup{I}_{g}}
\newcommand{\im}{\textup{i}}
\newcommand{\Ip}{\I_{+}}

\newcommand{\K}{\mathcal{K}}

\newcommand{\Mg}{{\mathfrak{M}}}

\newcommand{\nabmy}[2]{\nabla^{(#1)}_{#2}}
\newcommand{\nabmyt}[2]{\widetilde{\nabla}^{(#1)}_{#2}}
\newcommand{\nablat}{\widetilde{\nabla}}
\newcommand{\nablahat}{\widehat{\nabla}}

\newcommand{\nut}{\widetilde{\nu}}

\newcommand{\omegat}{\widetilde{\omega}}
\newcommand{\Omo}{\Omega_{0}}
\newcommand{\Omt}{\widetilde{\Omega}}

\newcommand{\Psihat}{\widehat{\Psi}}
\newcommand{\PsiM}{{\Psi}_{\Mg}}
\newcommand{\PsiMt}{\widetilde{\Psi}_{\Mg}}

\newcommand{\Schg}{\mathfrak{S}_{g}}

\newcommand{\Sg}{\mathcal{S}^{(g)}}

\newcommand{\Szero}{\mathcal{S}^{(0)}}
\newcommand{\SL}{\textup{SL}}

\newcommand{\tauK}{\tau_{\calK}}
\newcommand{\tpi}{2\pi \im} 

\newcommand{\vac}{\mathbbm{1}}

\newcommand{\wt}{\textup{wt}}

\newcommand{\Z}{\mathbb{Z}}
\newcommand{\zbar}{\overline{z}}
\newcommand{\Zg}{Z} 

\newcommand{\Zzero}{Z^{(0)}}

\newcommand{\edge}[2]{\xy(0,0)*[o]=<0.4pt>+{\cir<3pt>{}}="a"*+!R{#1\,}; (10,0)*[o]=<0.4pt>+{\cir<3pt>{}}="b"*+!L{\,#2}; \ar "b";"a";\endxy}

\newcommand{\sloop}[2]{
\xy (0,-2.5)*[o]=<0.5pt>+{\cir<3pt>{}}="b"*+!#2{#1\,};
\ar@(ul,ur) "b";"b"
\endxy}

\newcommand{\twocycle}[2]{\xy (0,0)*[o]=<0.4pt>+{\cir<3pt>{}}="a"*+!R{#1\,}; (10,0)*[o]=<0.4pt>+{\cir<3pt>{}}="b"*+!L{\,#2}; \ar@/^/ "b";"a";\ar@/^/ "a";"b";\endxy}

\newcommand{\degenchain}[2]{\xy(0,0)*{\cir<3pt>{}}="a"*+!#2{\,#1};\endxy}

\newtheorem{corollary}{Corollary}[section]

\newtheorem{lemma}{Lemma}[section]

\newtheorem{proposition}{Proposition}[section]
\newtheorem{remark}{Remark}[section]
\newtheorem{theorem}{Theorem}[section]

\title{Genus $g$ Virasoro Correlation Functions for Vertex Operator Algebras}
\date{}
\begin{document}
	\author{Michael P. Tuite}
\address{School of Mathematical and Statistical Sciences\\ 
	University of Galway, Galway H91 TK33, Ireland}
\email{michael.tuite@universityofgalway.ie}
\author[Michael Welby]{Michael Welby$^{\dagger}$}
\thanks{$^{\dagger}$Supported by an Irish Research Council Government of Ireland Postgraduate Scholarship.}
\email{michael.welby5@gmail.com}

\begin{abstract}
For a simple, self-dual, strong CFT-type vertex operator algebra (VOA) of central charge $c$, we describe  the Virasoro $n$-point correlation function on a genus $g$ marked Riemann surface in the Schottky uniformisation. We show that this $n$-point function determines the correlation functions for all Virasoro vacuum descendants. Using our recent work on genus $g$ Zhu recursion, we show that the Virasoro $n$-point function is determined by a differential operator $\calD_{n}$ acting on the genus $g$ VOA partition function normalised by the Heisenberg partition function to the power of $c$. We express $\calD_{n}$ as the sum of weights over certain  Virasoro graphs where the weights explicitly depend on $c$, the classical bidifferential of the second kind, the projective connection, holomorphic 1-forms and derivatives with respect to any $3g-3$ locally independent period matrix elements. We also describe the modular properties of $\calD_{n}$  under a homology base change.
\end{abstract}

\maketitle

\section{Introduction}
A vertex operator algebra (VOA) (e.g.~\cite{K,LL,MT1}) is an algebraic theory closely related to conformal field theory (CFT) in physics e.g. \cite{DFMS}. An essential ingredient in both theories is the Virasoro vector whose vertex operator modes generate a Virasoro algebra of some central charge $c$. Virasoro $n$-point correlation functions on genus zero and one Riemann surfaces have long been studied in the CFT and VOA literature. For instance, they occur in describing the Kac determinant and genus zero CFT Ward identities (e.g. \cite{DFMS}) and VOA modular differential equations arising from genus one Zhu recursion \cite{Z1}.
In this paper we apply recent results from \cite{TW1} concerning Zhu recursion theory for correlation functions on Riemann surfaces of genus $g\ge 2$ constructed by a Schottky uniformatisation. 
In particular, we consider the genus $g$ Virasoro $n$-point function from which all correlation functions for Virasoro vacuum descendants can be determined.
We show that this is expressible as an order $n$ differential operator acting on a suitably normalised genus $g$ partition function where the differential operator is explicitly described as a sums of certain weights of Virasoro graphs previously introduced in \cite{HT} for genus zero and one Virasoro $n$-point functions.  
Here the graph weights are defined in terms of some classical differential forms and differential operators with respect to $3g-3$ locally independent period matrix elements  on the  Riemann surface.

In Section 2 we briefly review the classical Schottky uniformisation of a genus $g$ marked Riemann surface $\Sg$~\cite{Fo,FK, Bo}, where we sew $g$ handles to a Riemann sphere, expressed in terms of $3g$ Schottky sewing parameters.
We review the classical bidifferential of the second kind, the projective connection, holomorphic 1-forms and the period matrix. These feature in some differential equations later in Section~2 and also enter into the description of Virasoro $n$-point functions in Section~4.
We review properties of the Bers   quasidifferential $(N,1-N)$-form $\Psi_{N}(x,y)$ \cite{Be1,Be2,TW2}. $\Psi_{N}$ and an associated spanning set of holomorphic $N$ forms appear in the genus $g$ Zhu recursion formula of \cite{TW1}. 
$\Psi_{2}(x,y)$ plays a crucial role in this paper since we exploit Zhu recursion for Virasoro vector $\omega$ insertions. 
Using $\Psi_{2}(x,y)$ and its associated  $2$-form spanning set, 
we define the fundamental differential operators $\nabla(x)$ and  $\nabla_{\bm{y}}^{(\bm{m})}(x)$  for variations in the Schottky parameters and local coordinates $\bm{y}=y_{1},\ldots, y_{n}$ on  $\Sg$. 
$\nabla_{\bm{y}}^{(\bm{m})}(x)$ maps any meromorphic form of weight $(\bm{m})=(m_{1},\ldots,m_{n})$ in $\bm{y}=y_{1},\ldots, y_{n}$ to a meromorphic form  of weight $(2,\bm{m})$ in $x,\bm{y}$ \cite{TW2}. 
 Section~2 also contains some refinements of \cite{TW2} where,  using global $\SL(2,\C)$ M\"obius invariance, we introduce differential operators $\nabla_{\Mg}(x)$ and  $\nabla_{\Mg,\bm{y}}^{(\bm{m})}(x)$ in terms of local variations in Schottky space $\Schg$  (the $\SL(2,\C)$ quotient of the Schottky parameter space) and local  coordinates $\bm{y}$ on  $\Sg$. We further show that $\nabla_{\Mg}(x)$ can be expressed in terms of variations of any local coordinates on moduli space $\Mg$. 
 Lastly, using these operators, we describe some differential equations for classical differentials \cite{TW1,TW2}.

Section~3 begins with a brief review of Vertex Operator Algebra (VOA) theory. 
For a VOA $V$ which is simple, self-dual ($V$ isomorphic to its dual module $V'$) of strong CFT-type, we define the partition function $Z_{V}$ and $n$-point correlation function $\calF_{V}(\bm{v},\bm{z})$ on a genus $g$ Riemann surface where $\bm{v}=v_{1},\ldots,v_{n}\in V$ are inserted at $\bm{z}=z_{1},\ldots,z_{n}$, respectively. 
We review genus $g$ Zhu recursion \cite{TW1} which is an expansion of a $(n+1)$-point correlation function
$\calF_{V}(u,x;\bm{v};\bm{z})$, for $u\in V$ quasiprimary of weight $N$,
  in terms of $n$-point functions dependent on $\bm{z}$ with universal coefficients given by $z_{k}$ derivatives of the Bers quasiform $\Psi_{N}(x,z_{k}) $ and holomorphic $N$-forms in $x$.  

In Section~4 we consider the Virasoro $n$-point correlation function 
$\calF_{V}(\bm{\omega,z})$ with $n$ insertions of the Virasoro vector $\omega$.
We show that $\calF_{V}(\bm{\omega,z})$ is a generating function for  all Virasoro vacuum descendant correlation functions.
Genus $g$ Zhu recursion implies a recursive conformal Ward identity for $\calF_{V}(\bm{\omega,z})$  involving the differential operators $\nabla(x)$ and $\nabla^{(\bm{m})}_{\bm{z}}(x)$ of Section~2 e.g.  $\calF_{V}(\omega,x)=\nabla(x)Z_{V}$. The main result of this paper is an explicit expression for the normalised $n$-point function $G_{n}(\bm{z}):=\calF_{V}(\omega,x)Z_{M}^{-c}$ where $M$ is the rank 1 Heisenberg VOA and $c$ the central charge of $V$.
We show that  $G_{n}(\bm{z})=\calD_{n}(\bm{z})\Theta_{V}$ for normalized partition function $\Theta_{V}:=Z_{V} Z_{M}^{-c}$ and where $\calD_{n}(\bm{z})$ is an order $n$  differential operator with respect to $3g-3$ locally independent period matrix elements. $\calD_{n}(\bm{z})$ is symmetric under permutations of $\bm{z}$ and is defined as the sum of weights of so-called Virasoro graphs \cite{HT} where the weights are defined in terms of $c$, the classical differentials of Section~2 and variations of the $3g-3$ independent period matrix elements.

 In Section~5 we consider a change of Riemann surface marking in the Schottky scheme described by the action of the symplectic modular group $\Sp(2g,\Z)$ on the given homology basis. We show that $\nabla_{\Mg}(x)$ and $ \nabla_{\Mg,\bm{y}}^{(\bm{m})}(x)$ are modular invariant whereas $\calD_{n}(\bm{z})$ is invariant up to a $c$-dependent Teichm\"uller  form-like automorphic factor. We conclude with some remarks about the significance of our results.

	\section{Differential Structures on Riemann Surfaces}
\label{sec:Genusg}
\subsection{Notational conventions}
Define the following indexing sets for integers $g,N\ge 1$ 
\begin{align}
	\label{eq:ILconv}
	\I:=\{-1,\ldots,- g,1,\ldots,g\},\quad
	\Ip:=\{1,\ldots,g\},\quad 
	\calL_{N}:=\{0,1,\ldots,2N-2\}.
\end{align}
For functions $f(x)$ and $g(x,y)$ and integers $i,j\ge 0$ we define
\begin{align*}
	&f^{(i)}(x):=\del^{(i)} f(x) :=\del_{x}^{(i)} f(x):= \frac{1}{i!}\frac{\del^{i}}{\del x^{i}}f(x),
	\\
	&g^{(i,j)}(x,y):=\del_{x}^{(i)} \del_{y}^{(j)}  g(x,y).
\end{align*} 

\subsection{The Schottky uniformisation of a Riemann surface}
\label{subsec:Schottky}
Consider a compact marked Riemann surface $\Sg$ of genus $g$, e.g.~\cite{FK,Mu1,Fa,Bo},  
with canonical homology basis $\alpha_{a}$, $\beta_{a}$ for $a\in\Ip$. 
We review the  construction of a genus $g$ Riemann surface $\Sg$ using the Schottky uniformisation where we sew $g$ handles to the Riemann sphere $\Szero\cong\Chat:=\C\cup \{\infty\}$ e.g.~\cite{Fo, Bo}. Every Riemann surface can be (non-uniquely) Schottky uniformised~\cite{Be2}.

For  $a\in\I$  let $\calC_{a}\subset \Szero$ be $2g$ non-intersecting Jordan curves where $z\in \calC_{a}$, $z'\in \calC_{-a}$ for $a\in\Ip$ are identified by a sewing relation
\begin{align}\label{eq:SchottkySewing}
	\frac{z'-W_{-a}}{z'-W_{a}}\cdot\frac{z-W_{a}}{z-W_{-a}}=q_{a},\quad a\in\Ip,
\end{align}
for $q_{a}$ with $0<|q_{a}|<1$ and $W_{\pm a}\in\Chat$. 
Thus  $z'=\gamma_{a}z$  for $a\in\Ip$ with 
\begin{align*}
	\gamma_{a}:=\sigma_{a}^{-1}
	\begin{pmatrix}
		q_{a}^{1/2} &0\\
		0 &q_{a}^{-1/2}
	\end{pmatrix}
	\sigma_{a},\quad 
	\sigma_{a}:=(W_{-a}-W_{a})^{-1/2}\begin{pmatrix}
		1 & -W_{-a}\\
		1 & -W_{a}
	\end{pmatrix}.
\end{align*}
The points $\sigma_{a}(W_{-a})=0$ and $\sigma_{a}(W_{a})=\infty$ are, respectively, attractive and repelling fixed points of $Z\rightarrow Z'=q_{a}Z$ for  $Z=\sigma_{a} z$ and $Z'=\sigma_{a} z'$.
$W_{-a}$ and  $W_{a}$  are the corresponding fixed points for $\gamma_{a}$. We identify the standard homology cycles $\alpha_{a}$  with $\calC_{-a}$ and  $\beta_{a}$ with a path connecting  $z\in  \calC_{a}$ to $z'=\gamma_{a}z\in  \calC_{-a}$.

The genus $g$ Schottky group $\Gamma$  is the free group with generators $\gamma_{a}$ for $a \in \Ip$.
Define $\gamma_{-a}:=\gamma_{a}^{-1}$. The independent elements of $\Gamma$ are reduced words of length $k$
of the form $\gamma=\gamma_{a_{1}}\ldots \gamma_{a_{k}}$ where $a_{i}\neq -a_{i+1}$ for each $i=1,\ldots,k-1$. 
We let $\Lambda(\Gamma)$ denote the  limit set\footnote{Note that for $g=1$, $\Lambda(\Gamma)$ is the empty set.}  of $\Gamma$ i.e. the set of limit points of the action of $\Gamma$ on $\Chat$. Then $\Sg\simeq\Omo/\Gamma$ where $\Omo:=\Chat-\Lambda(\Gamma)$.

Define $w_{a}:=\gamma_{-a}(\infty)$. Using~\eqref{eq:SchottkySewing} we find 
\begin{align}\label{eq:wadef}
	w_{a}=\frac{W_{a}-q_{a}W_{-a}}{1-q_{a}},\quad a\in\I,
\end{align}
where we define $q_{-a}:=q_{a}$.
Then~\eqref{eq:SchottkySewing} is equivalent to
\begin{align}\label{eq:SchottkySewing2}
	(z'-w_{-a})(z-w_{a})=\rho_{a},\quad a\in\Ip,
\end{align}
with 
\begin{align}\label{eq:rhoadef}
	\rho_{a}:=-\frac{q_{a}(W_{-a}-W_{a})^{2}}{(1-q_{a})^{2}}
	=-\frac{q_{a}(w_{-a}-w_{a})^{2}}{(1+q_{a})^{2}}.
\end{align}
Equation~\eqref{eq:SchottkySewing2} implies
\begin{align*}
	\gamma_{a}z=w_{-a}+\frac{\rho_{a}}{z-w_{a}}.
\end{align*}
It is convenient (but not necessary) to choose the Jordan curve $\calC_{a}$ to be the  boundary of the disc $\Delta_{a}$ with centre $w_{a}$ and radius $|\rho_{a}|^{\frac{1}{2}}$. Then 
$\gamma_{a}$ maps the exterior (interior) of $\Delta_{a}$  to the interior (exterior) of $\Delta_{-a}$ since
\begin{align*}
	|\gamma_{a}z-w_{-a}||z-w_{a}|=|\rho_{a}|.
\end{align*}
Furthermore,  the discs $\Delta_{a}$, $\Delta_{b}$ are non-intersecting if and only if
\begin{align}
	\label{eq:JordanIneq}
	|w_{a}-w_{b}|>|\rho_{a}|^{\frac{1}{2}}+|\rho_{b}|^{\frac{1}{2}},\quad \forall \;a\neq b.
\end{align} 
We define $\Cg$ to be the set 
$\{ (w_{a},w_{-a},\rho_{a})|\,a\in\Ip\}\subset \C^{3g}$ satisfying~\eqref{eq:JordanIneq}. We refer to $\Cg$ as the Schottky parameter space.

The cross ratio~\eqref{eq:SchottkySewing} is M\"obius  invariant for $\sigma =\left(\begin{smallmatrix}A&B\\C&D\end{smallmatrix}\right)\in\SL_{2}(\C)$ with $(z,z',W_{a},q_{a})\rightarrow(\sigma z,\sigma z',\sigma W_{a},q_{a})$ giving a global M\"obius $\SL_{2}(\C)$ action on $\Cg$ as follows
\begin{align}
	\label{eq:Mobwrhoa}
	\sigma:(w_{a},\rho_{a})\mapsto & 
	\left(	\frac { \left( Aw_{a}+B \right)  \left( Cw_{-a}+D \right) -\rho_{a}
		\,AC}{ \left( Cw_{a}+D \right)  \left( Cw_{-a}+D \right) -\rho_{a}\,{
			C}^{2}},
	{\frac {\rho_{a}}{ \left(  \left( Cw_{a}+D \right)  \left( Cw_{-a}+D
			\right) -\rho_{a}\,{C}^{2} \right) ^{2}}}\right).
\end{align}
Furthermore, $\sigma z'=\sigma \gamma_{a} \sigma^{-1}(\sigma z)$ so that $\Gamma$ is mapped to the conjugate Schottky group $\sigma \Gamma \sigma^{-1}$. 
We define Schottky space as $\mathfrak{S}_{g}:=\Cg/\SL_{2}(\C)$ which provides a natural covering space for the moduli space $\Mg$ of genus $g$ Riemann surfaces (of complex dimension 1 for $g=1$ and $3g-3$ for $g\ge 2$).  
We exploit the $\Cg$  parametrisation throughout  because the sewing relation~\eqref{eq:SchottkySewing2} is more  readily implemented in the theory of vertex operators.

\subsection{Some classical differentials on a Riemann surface}
Let $\Sg$  be a marked compact genus $g$ Riemann surface with canonical homology basis $\alpha_{a}, \beta_{a}$ for $a\in \Ip$.
The meromorphic bidifferential form of the second kind is a unique symmetric form  
\cite{Mu1,Fa}
\begin{align}\label{eq:omega}
	\omega(x,y)=\left(\frac{1}{(x-y)^{2}}+\text{regular terms}\right)\, dxdy,
\end{align}%
for local coordinates $x,y$ where
$
\oint_{\alpha_{a}}\omega(x,\cdot )=0$ for all $a\in \Ip$. It follows that
\begin{align}
	\nu_{a}(x):=\oint\limits_{\beta_{a}}\omega(x,\cdot ),  \quad a\in \Ip,
	\label{eq:nu}
\end{align}%
is a holomorphic 1-form normalised by
$\oint\limits_{\alpha_{a}}\nu_{b}=\tpi\, \delta_{ab}$. $\{\nu_{a}(x)\}$ is a basis for the space of holomorphic 1-forms $\calH_{1}$. 
The $g \times g$ symmetric period matrix is given by $\Omega=\frac{1}{\tpi}\tau$ where
\begin{align}
	\tau_{ab}:=\oint\limits_{\beta_{a}}\nu_{b},\quad \quad a,b\in \Ip.
	\label{eq:period}
\end{align}%
The projective connection $s(x)$ is defined by
\begin{align}\label{eq:projcon}
	s(x):=6\lim_{y\rightarrow x}\left(\omega(x,y)-\frac{dxdy}{(x-y)^{2}}\right).
\end{align} 
The Bers quasiform of weight $(N,1-N)$ for $g\ge 2$ and $N\ge2$ is defined by the following Poincar\'e series \cite{Be1,Be2,McIT,TW2} 
\begin{align}
	\label{eq:PsiNdef}
	\Psi_{N}(x,y):&=
	\sum_{\gamma\in\Gamma} \Pi_{N}(\gamma x,y),\quad x,y\in \Omo,
	\\
	\label{eq:PipiN}
	\Pi_{N}(x,y):&=\Pi_{N}(x,y;\bm{A}):=\frac{1}{x-y}\prod_{\ell\in\calL_{N}}\frac{y-A_{\ell}}{x-A_{\ell}}dx^{N}dy^{1-N},
\end{align}
for index set $\calL_{N}$ of~\eqref{eq:ILconv} and
where $\bm{A}:=A_{0},\ldots,A_{2N-2}\in\Lambda(\Gamma)$ are distinct limit points of $\Gamma$.
$\Pi_{N}$ is globally M\"obius invariant with 
\begin{align}\label{eq:PiN_Mobius}
	\Pi_{N}(\sigma x, \sigma y;\bm{\sigma A})=\Pi_{N}(x,y;\bm{A}),
\end{align} 
for all $\sigma\in\SL_{2}(\C)$.
$\Psi_{N}(x,y)$ is meromorphic  for $x,y\in \Omo$ with simple poles  of residue one at $y=\gamma x$ for all $\gamma\in\Gamma$. It is an $N$-differential in $x$ since
\begin{align*}
	\Psi_{N}(\gamma x,y) = \Psi_{N}(x,y),\quad \gamma\in\Gamma,
\end{align*} 
by construction, and is a quasiperiodic $(1-N)$-form  in $y$ with  \cite{TW2}
\begin{align}\label{eq:PsiCocyc}
	\Psi_{N}( x,\gamma_{a} y) - \Psi_{N}(x,y)=-\sum_{\ell\in\calL_{N}}\Theta_{N,a}^{\ell}(x)
	(y-w_{a})^{\ell}dy^{1-N},
\end{align}
for each Schottky group generator $\gamma_{a}$ and where 
$	\{\Theta_{N,a}^{\ell}(x)\}_{a\in\Ip}^{\ell\in\calL_{N}}$
spans the vector space $\calH_{N}$ of holomorphic $N$-forms. We note that  $\dim\calH_{N}=(g-1)(2N-1)$ for $g,N\ge 2$ by the Riemann-Roch Theorem.

For $N=1$ and $g\ge 2$  we define $\Psi_{1}(x,y):=\sum_{\gamma\in\Gamma}\Pi_{1}(\gamma x,y)$ with 
\begin{align*}
	\Pi_{1}(x,y):=\left(\frac{1}{x-y}-\frac{1}{x-A_{0}}\right)dx,
\end{align*}
where here $A_{0}\in \Omega_{0}$. Then we find \cite{Bu}
\begin{align}
	\label{eq:omPoincare2}
	\omega(x,y)=\Psi_{1}^{(0,1)}(x,y)dy.
\end{align}
Hence $\Psi_{1}(x,y)=\omega_{y-A_{0}}(x)=\int_{A_{0}}^{y}\omega(x,\cdot)$, the classical differential of the third kind which is not a Bers quasiform since it has simple poles both at $x=y$ and $x=A_{0}$.
Lastly, for $N\ge 1$, we define a canonical symmetric meromorphic 
$(N,N)$-form $\omega_{N}(x,y)$ with a pole of order $2N$ at $x=y$, generalising \eqref{eq:omPoincare2} as follows:
\begin{align}
	\label{eq:omegaN}
	\omega_{N}(x,y):=
	\Psi_{N}^{(0,2N-1)}(x,y)dy^{2N-1}
	=\sum_{\gamma \in\Gamma} \dfrac{d(\gamma x)^{N}dy^{N}}{(\gamma x-y)^{2N}}.
\end{align}

\subsection{Differential operators on meromorphic forms}
We review some first order differential operators $\D^{p}$, $\D^{p,(\bm{m})}_{\bfy}$, $\nabla(x)$ and $\nabla_{\bm{y}}^{(\bm{m})}(x)$ introduced
in~\cite{TW2} and developed in a VOA context in~\cite{TW1}. 
We first define the following basis for the tangent space $T(\Cg)$:
\begin{align}
	\label{eq:delael}
	\del_{a}^{0}:=\del_{w_{a}},\quad
	\del_{a}^{1}:=\rho_{a} \del_{\rho_{a}},\quad
	\del_{a}^{2}:=\rho_{a} \del_{w_{-a}}, \mbox{  for } a\in\Ip.
\end{align}
The generator $\D^{p}$ of any global M\"obius transformation \eqref{eq:Mobwrhoa} is determined by a quadratic polynomial $p(z)$ where \cite{TW2,TW1}
\begin{align}
	\label{eq:DpMobius}
	\D^{p}:&=\sum_{a\in\I}p(W_{a})\del_{W_{a}}		
	=
	\sum_{a\in\I}\left(  
	p(w_{a})\del_{w_{a}}+p^{(1)}(w_{a})\rho_{a}\del_{\rho_{a}}
	+p^{(2)}(w_{a}) \rho_{a}\del_{w_{-a}}\right)
	\\
	\notag
	&=\sum_{a\in\Ip}\sum_{\ell\in\calL_{2}}p_{a}^{\ell}\del_{a}^{\,\ell},
\end{align}
 with\footnote{Note that $p_{a}^{\ell}$ is defined with opposite sign in \cite{TW1,TW2}.} 
 $
 p_{a}^{\ell} :=p^{(\ell)}(w_{a})+ \rho_{a}^{1-\ell}p^{(2-\ell)}(w_{-a})$.
We next define a differential operator mapping differentiable functions on $\Cg$  to $\calH_{2}$:
\begin{align}\label{eq:nabladef}
	\delx:=\sum_{a\in\Ip}\sum_{\ell\in\calL_{2}}\Theta_{2,a}^{\ell}(x)\del_{a}^{\,\ell},
\end{align}
for $ \calH_{2}$ spanning set $\{\Theta_{2,a}^{\ell}(x)\}$ associated with quasi-periodicity  $\Psi_{2}(x,y)$ in \eqref{eq:PsiCocyc}. 
$\Psi_{2}(x,y)$  depends on the choice of the 3 limit points $\bm{A}= A_{0},A_{1},A_{2}$  in \eqref{eq:PipiN}. Any pair of limit points $\bm{A}$ and $\bm{\Ahat}$ are related by a M\"obius map so that $\Psi_{2}(x,y)$ is unique up to global M\"obius maps using \eqref{eq:PiN_Mobius}.
This has important consequences for the operator $\nabla(x)$ and its generalisation below. Let $\Psihat_{2}(x,y)$ be the Bers quasiform for alternative limit points $\bm{\Ahat}$. Then
$\Psihat_{2}(x,y)-\Psi_{2}(x,y)$ is holomorphic with
\begin{align}\label{eq:Psihat}
	\Psihat_{2}(x,y)=\Psi_{2}(x,y) +\sum_{r=1}^{3g-3}\Phi_{r}(x)p_{r}(y)dy^{-1}, 
\end{align}
for some $\calH_{2}$ basis $\{ \Phi_{r}(x)\}$ and quadratic polynomials $p_{r}$. For the corresponding $\calH_{2}$ spanning set $\{ \widehat{\Theta}_{2,a}^{\ell}(x)\}$ from \eqref{eq:PsiCocyc} and with $(p_{r})_{a}^{\ell}$ of \eqref{eq:DpMobius}  we find 
\begin{align}\label{eq:Thetahat}
	\widehat{\Theta}_{2,a}^{\ell}(x)=\Theta_{2,a}^{\ell}(x)
	+\sum_{r=1}^{3g-3} (p_{r})_{a}^{\ell}\Phi_{r}(x),
\end{align}
and differential operator
\begin{align}\label{eq:nablahat}
	\nablahat(x)=\nabla(x)+\sum_{r=1}^{3g-3} \Phi_{r}(x)\D^{p_{r}},
\end{align}
for  M\"obius generators $\D^{p_{r}}$. 
Thus $\nabla(x)$ determines a unique tangent vector field $\delMg(x)$, independent of $\bm{A}$, on the Schottky tangent space
$T(\Schg)\cong T(\Mg)$. We therefore find:
\begin{lemma}\label{lem:nablaM}
Let $F(\bm{\eta})$ be a differentiable function on $\Mg$ for any local coordinates $\bm{\eta}:=\eta_{1},\ldots,\eta_{3g-3}$. Then
\begin{align}\label{eq:nablaG}
		\nabla(x)F(\bm{\eta})=	\nabla_{\Mg}(x)F(\bm{\eta}).
\end{align}		
\end{lemma}	
Let $\calM^{(\bfm)}$ denote the space of meromorphic forms  $H^{(\bfm)}(\bfy)$ in 
$n$ variables $\bfy:=y_{1},\ldots,y_{n}$ of weight $(\bfm)$ for $\bfm:=m_{1},\cdots ,m_{n}$ i.e.
$H^{(\bfm)}(\bfy)=h(\bfy)dy_{1}^{m_{1}} \ldots dy_{n}^{m_{n}}$ for some meromorphic function $h(\bfy)$. 
We define the following differential operator on $\calM^{(\bfm)}$:
\begin{align}
	\label{eq:nabla_xym}
	\nabmy{\bfm}{\bfy}(x):=
	\nabla(x)
	+\sum_{k=1}^{n}\Big(\Psi_{2}(x,y_{k})\,d_{y_{k}}+m_{k} d_{y_{k}}\left(\Psi_{2}(x,y_{k})\right)\Big),
\end{align}
where $d_{y}(f(y)):=\partial_{y}f(y)dy$. 
\begin{remark}\label{rem:delyx}
	We note that $\nabmy{\bfm}{\bfy}(x)$ satisfies a Leibniz product rule (with adjusted weights)   e.g. 
	for $G(y)\in \calM^{(m)}$ and $H(z)\in \calM^{(n)}$ we have
	\[
	\nabmy{m,n}{y,z}(x)(G(y)H(z))=
	\left(\nabmy{m}{y}(x)G(y)\right)H(z)
	+G(y)\nabmy{n}{z}(x)H(z). 
	\]
	The definition \eqref{eq:nabla_xym} is unambiguous if any $\bm{y}$ variables are equal e.g. for $G(y,z)\in \calM^{(m,n)}$ regular at $y=z$, then $G(y,y)\in\calM^{(m+n)}$ with 
	\[
	\nabmy{m,n}{y,y}(x)G(y,y)=
	\nabmy{m+n}{y}(x)G(y,y).
	\]
\end{remark}
Choosing alternative limit points $\bm{\Ahat}$ so that \eqref{eq:Psihat} holds,  then \eqref{eq:nablahat} generalises to
\begin{align}\label{eq:nablahat}
	\nablahat^{(\bfm)}_{\bfy}(x)=\nabmy{\bfm}{\bfy}(x)+\sum_{r=1}^{3g-3}\Phi_{r}(x)\D^{(\bm{m}),p_{r}}_{\bfy}.
\end{align}
where 
\begin{align}
\D^{p,(\bm{m})}_{\bfy}:=\D^{p}
+\sum_{k=1}^{n}\left(p(y_{k})\partial_{y_{k}}+m_{k}p^{(1)}(y_{k})\right),
\label{eq:Dpy}
\end{align}
for any quadratic polynomial $p(y)$.
We find \cite{TW2}
\begin{proposition}\label{prop:Delmproperties}
Let $H^{(\bfm)}(\bfy)\in \calM^{(\bfm)}$. Then
	\begin{enumerate}
		\item [(i)] $\D^{p,(\bm{m})}_{\bfy}$ generates a global M\"obius transformation with
		$\D^{p,(\bm{m})}_{\bfy}H^{(\bfm)}(\bfy)=0$;
		\item[(ii)]
		$\nabmy{\bfm}{\bfy}(x)H^{(\bfm)}(\bfy)\in \calM^{(2,\bfm)}$, 
		the space of meromorphic forms of weight $(2,\bfm)$;
		\item[(iii)]
		Commutativity: $\nabmy{2,\bfm}{x,\bfy}(z)\nabmy{\bfm}{\bfy}(x)H^{(\bfm)}(\bfy)
		=\nabmy{2,\bfm}{z,\bfy}(x)\nabmy{\bfm}{\bfy}(z)H^{(\bfm)}(\bfy)$. 
	\end{enumerate}	
\end{proposition} 
Eqns~\eqref{eq:nablahat}, \eqref{eq:Dpy} and  Proposition~\ref{prop:Delmproperties}~(i),(ii) imply the following generalisation of Lemma~\ref{lem:nablaM}
\begin{lemma}\label{lem:nablayM}
Let $H^{(\bfm)}(\bfy)\in \calM^{(\bfm)}$. Then 
\begin{align*}
\nabmy{\bfm}{\bfy}(x)H^{(\bfm)}(\bfy)=\nabmy{\bfm}{\Mg,\bfy}(x)H^{(\bfm)}(\bfy),
\end{align*}
for
\begin{align}\label{eq:nablayMg}
	\nabmy{\bfm}{\Mg,\bfy}(x):=
	\nabla_{\Mg}(x)
	+\sum_{k=1}^{n}\Big( \PsiM(x,y_{k})d_{y_{k}}
	+m_{k} d_{y_{k}}\left(\PsiM(x,y_{k})\right)\Big),
\end{align} 
where $\PsiM(x,y)$ is an $\bm{A}$ independent weight $(2,-1)$ quasiform .
\end{lemma}
We apply Lemma~\ref{lem:nablayM} to some differential equations\footnote{Eqns \eqref{eq:nab_nu} and \eqref{eq:nab_ombidiff} appear in \cite{O} for a postulated weight-$(2,-1)$ quasiform $\PsiM$.} discussed in  \cite{TW1,TW2,O} where
$\nabla_{\bm{y}}^{(\bm{m})}(x)$ acts on the period matrix, 1-forms, bidifferential and projective connection of \eqref{eq:omega}--\eqref{eq:projcon}. Applying Lemma~\ref{lem:nablayM} these may be written as follows:
\begin{align}
	&\nabla_{\Mg} (x)\tau_{ab}=\nu_{a}(x)\nu_{b}(x),
	\label{eq:nab_Om}
	\\[2pt]
	&\nabla_{\Mg,y}^{(1)}(x)\,\nu_{a}(y)=\omega(x,y)\nu_a(x),
	\label{eq:nab_nu}
	\\[2pt]
	&\nabla^{(1,1)}_{\Mg,y_{1},y_{2}}(x)\,\omega(y_{1},y_{2})
	=\omega(x,y_{1})\omega(x,y_{2}),\label{eq:nab_ombidiff}
	\\[2pt]
	&\frac{1}{6}\nabla_{\Mg,y}^{(2)}(x)s(y)=
	\omega(x,y)^{2}-\omega_{2}(x,y),
	\label{eq:nab_som}
\end{align}
where $\omega_{2}(x,y)$ is the weight $(2,2)$ meromorphic form of \eqref{eq:omegaN}. Eqn.~\eqref{eq:nab_Om} is equivalent to Rauch's formula \cite{R}. Eqn.~\eqref{eq:nab_nu} implies an explicit formula for $\PsiM(x,y)$:
\begin{lemma} 	\label{lem:PsiM}
For all $a,b\in\Ip$ with $a\neq b$ we have
\begin{align}\label{eq:PsiM}
\PsiM(x,y)=\frac{
\omega(x,y)
\begin{vmatrix}
\nu_{a}(y) & \nu_{a}(x) \\
\nu_{b}(y) & \nu_{b}(x) 
\end{vmatrix}
-\begin{vmatrix}
	\nu_{a}(y) & \nabla_{\Mg} (x)\nu_{a}(y) \\
	\nu_{b}(y) & \nabla_{\Mg} (x)\nu_{b}(y) 
\end{vmatrix}
}
{\begin{vmatrix}
		\nu_{a}(y) & d_{y}\nu_{a}(y) \\
		\nu_{b}(y) & d_{y}\nu_{b}(y) 
\end{vmatrix}}.
\end{align}
\end{lemma}
\begin{proof}
From \eqref{eq:nab_nu} we find
$\begin{vmatrix}
	\nu_{a}(y) & \nabla_{\Mg,y}^{(1)} (x)\nu_{a}(y) \\
	\nu_{b}(y) & \nabla_{\Mg,y}^{(1)} (x)\nu_{b}(y) 
\end{vmatrix}
=
\omega(x,y)
\begin{vmatrix}
	\nu_{a}(y) & \nu_{a}(x) \\
	\nu_{b}(y) & \nu_{b}(x) 
\end{vmatrix}$ from which we solve for $\PsiM(x,y)$.
\end{proof}

\subsection{Bers Potentials and Variations of Moduli}	
Let $\lambda(z,\zbar) dz d\zbar$ be the Poincar\'e metric and $\langle \, ,\, \rangle$ be the associated non-singular positive definite Petersson product on $\calH_{2}$. 
A Bers potential function $f(z,\zbar)$ for $\Phi(z)=\phi(z)dz^{2}\in\calH_{2}$ is a differentiable function on $\Omo$ such that \cite{Be1,TW2} 
\begin{align}
	\label{eq:dF}
	&\frac{1}{\pi}\partial_{\zbar} f=\overline{\phi (z)}\,\lambda(z,\zbar)^{-1},
\end{align} 
where $\lim_{z\rightarrow 0}\left\vert z ^{2}f\left(z^{-1}\right)\right\vert <\infty$.
A potential	  $f(z)$ can be constructed from a Bers quasiform $\Psi_{2}(x,z)$ as follows 
\begin{align}\label{eq:fpotential}
	f(z)	=-\langle \Psi_{2}(\cdot,z), \Phi(\cdot) \rangle dz.
\end{align}
The  potential $f$ for $\Phi$ is unique up to an additive quadratic polynomial in $z$. This follows from global M\"obius invariance of \eqref{eq:SchottkySewing} and \eqref{eq:PiN_Mobius}.

The Ahlfors map \cite{A} describes a bijective antilinear map between $\calH_{2}$  and the moduli tangent space $T(\Mg)$. This map can be realized in the Schottky parameterization as follows \cite{TW2}. 
Choose local coordinates $\bm{\eta}:=\eta_{1},\ldots,\eta_{3g-3}$ on $\Mg$ giving a local $T(\Mg)$ basis $\{\partial_{\eta_{r}}\}$.
For a given choice for $r$, consider a small moduli deformation $\eta_{r}\rightarrow \eta_{r}+\varepsilon$ of the Riemann surface with corresponding  quasiconformal map e.g. \cite{GL}
\begin{align*}
	z\rightarrow w_{r}=z+\frac{\varepsilon}{\pi} f_{r}+O(\varepsilon^2),
\end{align*}
for some $f_{r}(z,\zbar)$. 
The Beltrami equation  \cite{GL}  implies that
\[
\mu_{r}:=\frac{1}{\pi}\partial_{\zbar}f_{r} ,
\]
gives a harmonic Beltrami differential i.e. $f_{r}$ is
a Bers potential for  $\Phi_{r}:=\overline{\mu}_{r}\lambda dz^{2}\in\calH_{2}$ for each $r$. Since the moduli are independent,
$\{\Phi_{r}\}_{r=1}^{3g-3}$ is a $\calH_{2}$-basis.

The deformed Riemann surface is uniformized by some  Schottky group $\Gamma^{\varepsilon}$  where for each Schottky generator  $\gamma_{a}\in\Gamma$ for $a\in\Ip$ we define a generator $\gamma^{\varepsilon}_{a}\in\Gamma^{\varepsilon}$ via the compatibility condition $
\gamma_{a}^{\varepsilon}w_{r}(z)=w_{r}(\gamma_{a} z)$. This implies for each $\gamma\in\Gamma$ that\footnote{Each $\gamma_{a}\in\Gamma$ depends on $3g$ Schottky parameters which we can locally describe as functions of the moduli $\{\eta_{s}\}$ and 3 global M\"obius parameters.} \cite{TW2}
\begin{align}\label{eq:delmgamz}
	d(\gamma z)^{-1}\partial_{\eta_{r}}(\gamma_{a} z)=\frac{1}{\pi}\Xi_{r}[\gamma_{a}](z),\quad a\in\Ip,
\end{align}
where $\Xi_{r}[\gamma_{a}](z)=f_{r}(\gamma_{a} z)d(\gamma_{a} z)^{-1}-f_{r}(z)dz^{-1}$ is an Eichler cocycle that uniquely determines a potential for $\Phi_{r}\in\calH_{2}$ via the bijective antilinear Bers map \cite{Be1,Be2,TW2}.
Eqn.~\eqref{eq:delmgamz} allows us to identify a canonical bijective antilinear  Ahlfors map between $\partial_{\eta_{r}}$ and $\Phi_{r}$ for each $r=1,\ldots, 3g-3$ in the Schottky scheme. 
Let $\{\Phi_{r}^{\vee}\}$ be the dual basis to $\{\Phi_{r}\}$ with respect to  the Petersson product. Then 
$\sum_{r=1}^{3g-3}\Phi^{\vee}_{r}(x)\partial_{\eta_{r}}\in T(\Mg)$ is independent of the choice of local coordinates. Recalling $\nabla_{\Mg}(x)$ 
of \eqref{eq:nablaG} we find
\begin{proposition}\label{prop:nablaM}
$\nabla_{\Mg}(x)=-\pi \sum_{r=1}^{3g-3}\Phi^{\vee}_{r}(x)\partial_{\eta_{r}}$.
\end{proposition}
\begin{proof}
From \eqref{eq:PsiCocyc}, \eqref{eq:delael} and \eqref{eq:nabladef} 
we find $\del_{b}^{\ell}(\gamma_{a}z)=-\delta_{ab}(z-w_{a})^{\ell}\partial_{z}(\gamma_{a}z)$ so that \cite{TW2}
	\begin{align}
		\label{eq:nablayt}
		d(\gamma_{a} z)^{-1}\delx(\gamma_{a} z)
		=\Psi_{2}(x,\gamma_{a} z)-\Psi_{2}(x,z),\quad a\in\Ip. 
	\end{align}
For $\Phi_{r}$ paired with $\partial_{\eta_{r}}$ under the canonical Ahlfors map we define
	\begin{align}\label{eq:delrtan}
		\partial_{r}:=\langle\nabla(\cdot), \Phi_{r}(\cdot)\rangle
		=\sum_{b\in\Ip}\sum_{\ell\in\calL_{2}}\langle\Theta_{2,b}^{\ell},\Phi_{r}\rangle\del_{b}^{\,\ell}\in T(\Cg).
	\end{align}
Then \eqref{eq:nablayt} implies 
	\begin{align}
		\label{eq:Xihat}
		d(\gamma_{a}z)^{-1}\partial_{r}(\gamma_{a}z)=
		\langle\Psi_{2}(\cdot,\gamma_{a}z),\Phi_{r}(\cdot)\rangle
		-\langle\Psi_{2}(\cdot,z),\Phi_{r}(\cdot)\rangle
		=-\widehat{\Xi}_{r}[\gamma_{a}](z),
	\end{align}
	where  $\widehat{\Xi}_{r}[\gamma_{a}](z)=\widehat{f_{r}}(\gamma_{a}z)d(\gamma_{a}z)^{-1}-\widehat{f_{r}}(z)dz^{-1}$ 
	is the Eichler cocycle for $\Phi_{r}$ with a potential $\widehat{f_{r}}$ obtained from \eqref{eq:fpotential}. Hence $f_{r}=\widehat{f_{r}}+p(z)$ for some quadratic $p(z)$ where $f_{r}$ is the $\Phi_{r}$ potential determined in \eqref{eq:delmgamz}. Comparing \eqref{eq:delmgamz} and \eqref{eq:Xihat} we  find 
	\begin{align*}
\left(\partial_{r}+\pi\partial_{\eta_{r}} \right)(\gamma_{a}z)
=p(\gamma_{a}z)-p(z)\partial_{z}(\gamma_{a}z)=\D^{p}(\gamma_{a}z), \quad a\in\Ip,
	\end{align*}
for a global M\"obius generator $\D^{p}$ of \eqref{eq:DpMobius}.
Thus $\partial_{r}+\pi\partial_{\eta_{r}}=\D^{p}$ and hence 
	\begin{align}\label{eq:tangents}
		\partial_{r}F(\bm{\eta})=-\pi\partial_{\eta_{r}}F(\bm{\eta}),
	\end{align} 
for any differentiable function $F$ on $\Mg$. Using Lemma~\ref{lem:nablaM} and the chain rule we find
\begin{align*}
\nabla_{\Mg}(x)F(\bm{\eta})=\nabla(x)F(\bm{\eta})=\sum_{s=1}^{3g-3}\Theta_{s}(x)\partial_{\eta_{s}}F(\bm{\eta}),
\end{align*}
for $\Theta_{s}(x):=\delx \eta_{s}\in\calH_{2}$. 
But 
\eqref{eq:delrtan} and \eqref{eq:tangents} imply $
		\langle\Theta_{s},\Phi_{r}\rangle =-\pi \delta_{rs}$	so that $\Theta_{s}(x)=-\pi \Phi^{\vee}_{s}(x)$ and the result follows.
\end{proof}
We may choose $3g-3$ locally independent components of the period matrix $\tau$ of \eqref{eq:period} as local coordinates on $\Mg$ (the Schottky problem). 
Let $\tauK:=\ldots,\tau_{ab},\ldots$ for $(a,b)\in\K$,  a $\tau$ label set of cardinality $3g-3$ for such a choice. 
Recalling the 1-form basis $\{\nu_{a}\}_{a\in\Ip}$ of \eqref{eq:nu} 
and defining $\partial_{ab}:=\partial_{\tau_{ab}}=\frac{1}{\tpi}\partial_{\Omega_{ab}}$ we find
\begin{lemma}\label{lem:nablaMperiod}
	Let $\tauK$ be local coordinates on $\Mg$ for some independent period matrix elements with label set $\K$. Then 
	\begin{align}	\label{eq:nablagtau}
		\nabla_{\Mg}(x)=
		\sum_{(a,b)\in\K}\nu_{a}(x)\nu_{b}(x)\partial_{ab},
	\end{align}
	where $\{\nu_{a}(x)\nu_{b}(x)\}_{(a,b)\in\K}$ is a $\calH_{2}$-basis which is the Petersson dual of the $\calH_{2}$-basis paired with the local $T(\Mg)$-basis $\{-\frac{1}{\pi}\partial_{ab}\}_{(a,b)\in\K}$ by the canonical Ahlfors map. 
\end{lemma} 
\begin{proof}
	Let $F(\tauK)$ be a locally differentiable function on $\Mg$.
Eqn.~\eqref{eq:nablagtau}  follows from the chain rule and \eqref{eq:nab_Om} where we find
	\begin{align*}
		\nabla_{\Mg}(x)F(\tauK)=
		\sum_{(a,b)\in\K}\left(\nabla_{\Mg}(x)\tau_{ab}\right)\partial_{ab}F(\tauK)
		=\sum_{(a,b)\in\K}\nu_{a}(x)\nu_{b}(x)\partial_{ab}F(\tauK).
	\end{align*}
	From Proposition~\ref{prop:nablaM} we find that $\{\nu_{a}(x)\nu_{b}(x)\}_{(a,b)\in\K}$ is a local $\calH_{2}$-basis dual to the $\calH_{2}$-basis paired with $\{-\frac{1}{\pi}\partial_{ab}\}_{(a,b)\in\K}$.
\end{proof} 
Finally, we remark that for any differentiable  function $F(\tau)$ of the full period matrix $\tau$ then Lemma~\ref{lem:nablaM}, the chain rule and \eqref{eq:nab_Om} directly imply that
\begin{align*}
	\nabla_{\Mg}(x)F(\tau)=
	\sum_{1\le a\le b\le g}
	\nu_{a}(x)\nu_{b}(x)\partial_{ab}F(\tau).
\end{align*}

\section{Vertex Operator Algebras and Genus $g$ Zhu Recursion}
\label{sec:VOAs}
\subsection{Vertex operator algebras} 
For indeterminates  $x,y$ we adopt the binomial expansion convention that for $m\in\Z$ 
\begin{align*}
	(x+y)^{m}=\sum_{k\ge 0}\binom{m}{k}x^{m-k}y^{k}.
\end{align*} 

We review some aspects of vertex operator algebras e.g.~\cite{K,FHL,LL,MT1}. A vertex operator algebra (VOA) is a quadruple $(V,Y(\cdot,\cdot),\vac,\omega)$ consisting of a graded vector space $V=\bigoplus_{n\ge 0}V_{n}$, with $\dim V_{n}<\infty$,  with two distinguished elements: the vacuum vector $\vac\in V_{0}$  and the Virasoro conformal vector $\omega\in V_{2}$.  
For each $v \in V$ there exists a vertex operator, a formal Laurent series in $z$, given by   
\begin{align*}
	Y(u,z)=\sum_{n\in\Z}u(n)z^{-n-1},
\end{align*}
for \emph{modes} $u(n)\in\End(V)$. 
For each  $u,v\in V$ we have $u(n)v=0$ for all $n\gg 0$, known as \emph{lower truncation}, and 
$u=u(-1)\vac $ and $u(n)\vac =0$ for all $n\ge 0$,
known as \emph{creativity}.
The vertex operators also obey  \emph{locality}:
\begin{align*}
	(x-y)^{N}[Y(u,x),Y(v,y)]=0,\quad N\gg 0.
\end{align*}
For the Virasoro conformal vector
\begin{align*}
	Y(\omega,z)=\sum_{n\in\Z}L(n)z^{-n-2},
\end{align*}
where the operators  $L(n)=\omega(n+1)$  satisfy the Virasoro  algebra
\begin{align*}
	[L(m),L(n)]=(m-n)L(m+n)+\frac{c}{2}\binom{m+1}{3}\delta_{m,-n}\Id_{V},
\end{align*}
for a constant \emph{central charge} $c\in\C$.   Vertex operators  satisfy the \emph{translation property}: 
\begin{align*}
	Y(L(-1)u,z)=\del Y(u,z).
\end{align*}
Finally, $V_{n}=\{v\in V:L(0)v=nv\}$ where
$v\in V_{n}$ is the \emph{(conformal) weight} $\wt(v)=n$. We quote a number of basic VOA properties e.g.~\cite{K,FHL,LL,MT1}. For $u\in V$ of weight $N$ we have
\begin{align*}
	u(j):V_{k}\rightarrow V_{k+N-j-1}.
\end{align*}
The commutator identity: for all $u,v\in V$ we have 
\begin{align*}
	[u(k),Y(v,z)]=\left (\sum_{j\ge 0}Y(u(j)v,z)\partial_{z} ^{(j)}\right)z^{k}.
\end{align*} 
The associativity identity: for each $u,v\in V$  there exists $M\ge 0$ such that
\begin{align*}
	(x+y)^{M}Y(Y(u,x)v,y)=(x+y)^{M}Y(u,x+y)Y(v,y).
\end{align*}
Associated with the formal M\"obius map $z\rightarrow {\rho}/{z}$, for  a given scalar $\rho\neq 0$,
we define an adjoint vertex operator \cite{FHL, L}
\begin{align*}
	Y_{\rho}^{\dagger}(u,z):=\sum_{n\in\Z}u_{\rho}^{\dagger}(n)z^{-n-1}=Y\left(e^{\frac{z}{\rho}L(1)}\left(-\frac{\rho}{z^{2}}\right)^{L(0)}u,\frac{\rho}{z}\right).
\end{align*}
We write $Y^{\dagger}(u,z)$ for the adjoint when $\rho=1$.
For quasiprimary $u$ (i.e. $L(1)u=0$) of weight $N$ we have
\begin{align}
	\label{eq:udagger}
	u_{\rho}^{\dagger}(n)=(-1)^{N}\rho^{n+1-N}u(2N-2-n),
\end{align}
e.g. $L_{\rho}^{\dagger}(n)=\rho^{n}L(-n)$.
A bilinear form $\langle \cdot,\cdot\rangle_{\rho}$ on $V$  is said to be invariant if
\begin{align*}
	\langle Y (u, z)v, w\rangle_{\rho} = \langle v, Y_{\rho}^{\dagger}
	(u, z)w\rangle_{\rho},\quad \forall\;u,v,w\in V.
\end{align*}
If $\rho=1$ then we omit the $\rho$ subscripts.
$\langle \cdot,\cdot\rangle_{\rho}$ is symmetric and $\langle u,v\rangle_{\rho}=0$ for $\wt(u)\neq\wt(v)$~\cite{FHL} with 
\begin{align*}
	\langle u,v\rangle_{\rho}=\rho^{N}\langle u,v\rangle,
	\quad N=\wt(u)=\wt(v).
\end{align*} 
We assume throughout this paper that $V$ is of strong CFT-type i.e. $V_{0} = \C\vac$ and $L(1)V= 0$. Then the bilinear form with normalisation $\langle \vac,\vac\rangle_{\rho}=1$ is unique~\cite{L}. We also assume that $V$ is simple and isomorphic to the contragredient $V$-module $V'$ \cite{FHL}. 
Then the bilinear form is non-degenerate~\cite{L}. We refer to this unique invariant non-degenerate bilinear form as the Li-Zamolodchikov (Li-Z) metric.

\subsection{Genus $g$ correlation functions}
Define the genus zero $n$-point (correlation) function for $\bm{v}:=v,\ldots,v_{n}$ inserted at $\bm{z}:=z,\ldots,z_{n}$, respectively,  by\footnote{The superscript $(0)$ on $\Zzero(\bm{v,z})$ refers to the genus.}
\begin{align*}
	\Zzero(\bm{v,z}):=\Zzero(\ldots;v_{k},z_{k};\ldots)=\langle \vac,\bm{Y(v,z)}\vac\rangle,
\end{align*}
for ($\rho=1$) Li-Z metric $\langle\cdot,\cdot\rangle$ and 
\begin{align*}
	\bm{Y(v,z)}:= Y(v_{1},z_{1}) \ldots Y (v_{n},z_{n}).
\end{align*}
$\Zzero(\bm{v,z})$ can be extended to a rational function 
in $\bm{z}$ in the domain $|z_{1}|>\ldots >|z_{n}|$.  

We next define genus $g$ correlation functions in terms of certain infinite sums of genus zero correlation functions based on a formal version of the Schottky sewing scheme.
For each $a\in\Ip$, let $\{b_{a}\}$  denote a homogeneous  $V$-basis and let $\{\bbar _{a}\}$ be the  dual basis with respect to the Li-Z metric $\langle \cdot,\cdot\rangle$  i.e. with $\rho=1$.  Define
\begin{align}
\label{eq:bbar}
b_{-a}:=\rho_{a}^{\wt(b_{a})}\bbar _{a},\quad a\in\Ip,
\end{align}
for formal $\rho_{a}$ (later  identified with a Schottky sewing parameter). 
Then $\{b_{-a}\}$ is a dual basis for the Li-Z metric $\langle \cdot,\cdot\rangle_{\rho_{a}}$ with adjoint (cf.~\eqref{eq:udagger})
\begin{align*}
u^{\dagger}_{\rho_{a}}(m)=(-1)^{N}\rho_{a}^{m+1-N}u(2N-2-m),
\end{align*}
for $u$ quasiprimary  of weight $N$.
Let $\bm{b}_{+}=b_{1}\otimes\ldots \otimes b_{g}$ denote an element of a $V^{\otimes g}$-basis. Let $w_{a}$ for $a\in\I$ be $2g$ formal variables (later identified with the canonical Schottky parameters).
Consider the genus zero rational $2g$-point function
\begin{align*}
\Zzero(\bm{b,w})=\Zzero(b_{1},w_{1};b_{-1},w_{-1};\ldots;b_{g},w_{g};b_{-g},w_{-g})
,
\end{align*}
for  
$\bm{b,w}=b_{1},w_{1},b_{-1},w_{-1},\ldots,b_{g},w_{g},b_{-g},w_{-g}$. 
Define the genus $g$ partition function by
\begin{align}\label{eq:Zg}
\Zg_{V}:=\Zg_{V}(\bm{w,\rho})
=\sum_{\bm{b_{+}}}\Zzero(\bm{b,w}),
\end{align}
for $\bm{w,\rho}=w_{1},w_{-1},\rho_{1},\ldots,w_{g},w_{-g},\rho_{g}$  and 
where the sum is over any basis $\{\bm{b}_{+}\}$ of $V^{\otimes g}$.
This definition is motivated by the sewing relation~\eqref{eq:SchottkySewing2} and ideas in~\cite{MT2,MT3, T1,TW2}. This is similar to the sewing analysis  employed in \cite{Z2, C, DGT,G}. We suppress the genus superscript label $(g)$ except for genus zero. The genus $g$ partition function is formally M\"obius invariant where for $\D^{p}$ of \eqref{eq:DpMobius} we find \cite{TW1}
\begin{proposition}\label{prop:DpZ}
	$\D^{p}\Zg_{V}=0$ for any quadratic polynomial $p$.
\end{proposition}
Define the genus $g$ formal $n$-point function for $v_{1},\ldots,v_{n}\in V$ inserted at $z_{1},\ldots,z_{n}$ by
\begin{align*}
\Zg_{V}(\bm{v,z}):=\Zg_{V}(\bm{v,z};\bm{w,\rho})
=
\sum_{\bm{b}_{+}}\Zzero(\bm{v,z};\bm{b,w}),
\end{align*}
for rational genus zero $(n+2g)$-point functions 
\begin{align*}
\Zzero(\bm{v,z};\bm{b,w})=\Zzero(v_{1},z_{1};\ldots;v_{n},z_{n};b_{-1},w_{-1};\ldots;b_{g},w_{g}).
\end{align*}
 We also define the corresponding genus $g$ formal $n$-point correlation differential form
\begin{align*}
\calF_{V}(\bm{v,z}):=Z_{V}(\bm{v,z})\bm{dz^{\wt(v)}}.
\end{align*}
From \eqref{eq:bbar} we note that
	\begin{align*}
		\calF_{V}(\bm{v,z})=
		\sum_{n_{1},\ldots,n_{g}\ge 0}\rho_{1}^{n_{1}}\ldots \rho_{g}^{n_{g}}\sum_{b_{1}\in V_{n_{1}}}\ldots \sum_{b_{g}\in V_{n_{g}}} 
		\Zzero(\bm{v,z};b_{1},w_{1},\ldots,\bbar_{g},w_{-g})\bm{dz^{\wt(v)}}.
	\end{align*}
Thus the expansion of $\calF_{V}(\bm{v,z})$ in $\bm{\rho}$ to order  $\rho_{1}^{n_{1}}\ldots \rho_{g}^{n_{g}}$ has rational coefficient functions of $w_{a}$ and $\bm{z}$ and is convergent for all $(w_{\pm a},\rho_{a})\in\Cg$ and for all $z_{i}\neq z_{j}\in\Omega_{0}$ for $i\neq j$.

\begin{remark}\label{rem:Gui}
	If $V$ is $C_{2}$-cofinite then Gui's Theorem~13.1 \cite{G} implies that $\calF_{V}(\bm{v,z})$ is absolutely and locally uniformly convergent in the Schottky sewing domain. Here we treat $\calF_{V}(\bm{v,z})$ formally since Zhu recursion does not require $C_{2}$-cofiniteness.
\end{remark}
The $\bm{\rho}$ expansion of $Z_{V}$ has rational coefficient functions of $w_{a}$ convergent for all $(w_{\pm a},\rho_{a})\in\Cg$. Recalling, \eqref{eq:wadef} and \eqref{eq:rhoadef} relating $\rho_{a},w_{\pm a}$ to the original Schottky parameters, we may also consider the expansion of $Z_{V}$ in $\bm{q}:=q_{1},\ldots,q_{g}$. We then find
\begin{proposition}\label{prop:ZVqexp}
	The expansion of $Z_{V}$ to any finite order  $q_{1}^{n_{1}}\ldots q_{g}^{n_{g}}$ 
	is convergent on Schottky space $\Schg$. 
\end{proposition}
\begin{proof}
Since  $w_{a}=W_{a}+O(q_{a})$ and $\rho_{a}=-(W_{-a}-W_{a})^{2}q_{a}+O(q_{a}^{2})$ 
we may expand $Z_{V}$ in $\bm{q}$  to order  $q_{1}^{n_{1}}\ldots q_{g}^{n_{g}}$ from its $\bm{\rho}$ expansion to order  $\rho_{1}^{n_{1}}\ldots \rho_{g}^{n_{g}}$. Furthermore, the coefficients in the $\bm{q}$ expansion are convergent rational functions of $W_{\pm a}$ on $\Cg$. 
Proposition~\ref{prop:DpZ} implies that the coefficient functions are M\"obius invariant since $(W_{a},q_{a})\rightarrow(\sigma W_{a},q_{a})$ for  $\sigma\in\SL_{2}(\C)$. Thus the result follows.
\end{proof}

\subsection{Genus $g$ Zhu recursion}
We review the genus $g$ correlation function Zhu recursion formula\footnote{Zhu recursion for $V$-modules is discussed in ref.~\cite{TW1}.} for a VOA $V$ \cite{TW1}. This generalises the original Zhu \cite{Z1} recursion at genus zero and one. Zhu recursion is essentially a formal version of a residue expansion for meromorphic forms on Riemann surfaces  \cite{TW1, TW2}.
 \begin{theorem}\label{theor:ZhuGenusg}
	Let $V$ be a simple VOA of strong CFT-type with $V$ isomorphic  to $V'$.
The genus $g$ correlation form for quasiprimary $u$ of weight $N\ge 1$ inserted at $x\in\Sg$ and  $v_{1},\ldots,v_{n}\in V$ inserted at $z_{1},\ldots,z_{n} \in\Sg$, respectively, satisfies
\begin{align}
	\calF_{V}(u,x;\bm{v,z})&=\sum_{a\in\Ip}\sum_{\ell\in \calL_{N}} \Theta_{N,a}^{\ell}(x)\Res_{a}^{\ell}\calF_{V}(u;\bm{v,z})
	\label{eq:ZhuGenusg}
	\\
	\notag
	&
	+\sum_{k=1}^{n}\sum_{j\ge 0}\Psi_{N}^{(0,j)}(x,z_{k})\calF_{V}(\ldots;u(j)v_{k},z_{k};\ldots)\,dz_{k}^{j},
\end{align}
where 
$\Res_{a}^{\ell}\calF_{V}(u;\bm{v,z}):=
\Res_{x-w_{a}}\left(x-w_{a}\right)^{\ell}\calF_{V}(u,x;\bm{v,z})$, 
$\Psi_N(x,z)$ is the Bers quasiform~\eqref{eq:PsiNdef} and
$\{\Theta_{N,a}^{\ell}(x)\}$  of \eqref{eq:PsiCocyc} spans $\calH_{N}$.
\end{theorem}

\section{Genus $g$ Virasoro Correlation Functions}\label{sec:Vir}
\subsection{The Virasoro Ward Identity}
We now apply Theorem~\ref{theor:ZhuGenusg} to compute all Virasoro correlation functions  $\calF_{V}(\bm{\omega,z})$ where the Virasoro vector $\omega$  is inserted at $\bm{z}=z_{1},\ldots,z_{n}$. 
We find $\Res_{a}^{\ell}\calF_{V}(\omega)=\partial_{a}^{\ell}Z_{V}$ for $\partial_{a}^{\ell}$ of \eqref{eq:delael} so that~\cite{TW1}
\begin{align}\label{eq:Fom}
\calF_{V}(\omega,x)=\nabla(x)Z_{V},
\end{align}
with $\nabla(x)$ of \eqref{eq:nabladef}. 
For the rank 1 Heisenberg VOA $M$, we may also compute $\calF_{M}(\omega,x)$ in an alternative way to obtain the following differential equation for $Z_{M}$~\cite{TW1}
\begin{align}\label{eq:nablaZM}
\nabla(x)Z_M=\frac{1}{12}s(x)Z_{M},
\end{align}
for projective connection $s(x)$ of \eqref{eq:projcon}. In general we find that Theorem~\ref{theor:ZhuGenusg} implies the following Ward identity~\cite{TW1}:
\begin{proposition}\label{prop:Ward}
\begin{align}
	&\calF_{V}(\omega,x;\bm{\omega,z})
	=\nabla_{\bm{z}}^{(\bm{2})}(x)\calF_{V}(\bm{\omega,z})
	+\frac{c}{2}\sum_{k=1}^{n}\omega_{2}(x,z_{k})\calF_{V}(\ldots;\widehat{\omega,z_{k}};\ldots),
	\label{eq:VirWard}
\end{align}
for $n$-tuple  $\bm{2}=2,\ldots,2$  and where
the caret denotes omission of the $\omega$ insertion at $z_{k}$ and $\omega_{2}(xzy)$ is the symmetric meromorphic $(2,2)$-form of \eqref{eq:omegaN}.
\end{proposition}
\subsection{Virasoro $n$-point generating functions}
\begin{proposition}
\label{prop_Ggen}
$\calF_{V}(\bm{\omega,z})$ is symmetric in $z_i$ and is a generating function for all genus $g$ $n$-point correlation functions for Virasoro vacuum descendants.
\end{proposition}
\begin{proof}
	The proof follows that for corresponding results in \cite{HT,GT}. 
$\calF_{V}(\bm{z})$ is symmetric  in $z_1,\ldots,z_n$ by locality.  
Consider the $n$-point function for $n$ Virasoro vacuum descendants
 $v_i=L(-k_{i1})\ldots L(-k_{im_i})\vac$ for $k_{ij}\ge 2$ inserted at $\bm{z}$ 
\begin{align*}
Z_V\left(v_{1},z_1; \ldots; v_{n},z_n\right) 
=\sum_{\bm{b}_{+}} \langle\vac,Y(v_{1},z_1)\ldots Y(v_{n},z_n)\bm{Y(b,w)}\vac\rangle.
\end{align*}
Then $\langle\vac, Y(v_{1},z_1)\ldots Y(v_{n},z_n)\bm{Y(b,w)}\vac\rangle$ is the coefficient of $\prod_{i=1}^{n}\prod_{j=1}^{m_{i}}(x_{ij})^{k_{ij}-2}$  in
\[
\langle\vac,
Y(Y(\omega,x_{11})\ldots Y(\omega,x_{1m_1} )\vac,z_1)\ldots 
Y(Y(\omega,x_{n1})\ldots Y(\omega,x_{nm_n} )\vac,z_n)\bm{Y(b,w)}\vac\rangle.
\] 
Using associativity and lower truncation  (e.g. \cite{K,LL,MT1}), we find for $N\gg 0$ that
\begin{align*}
&\prod_{i=1}^{n}\prod_{j=1}^{m_{i}}(x_{ij}+z_i)^{N}
Y(Y(\omega,x_{11})\ldots Y(\omega,x_{1m_1} )\vac,z_1)\ldots 
Y(Y(\omega,x_{n1})\ldots Y(\omega,x_{nm_n} )\vac,z_n)
\\
&=\prod_{i=1}^{n}\prod_{j=1}^{m_{i}}(x_{ij}+z_i)^{N}
Y(\omega,z_1+x_{11})\ldots Y(\omega,z_1+x_{1m_1} )\ldots 
Y(\omega,z_n+x_{n1})\ldots Y(\omega,z_n+x_{nm_n}).
\end{align*}
 Thus $Z_V\left(v_{1},z_1; \ldots; v_{n},z_n\right) $
is the coefficient of $\prod_{i=1}^{n}\prod_{j=1}^{m_{i}}(x_{ij})^{k_{ij}-2}$ 
of the formal expansion of $Z_V\left(\omega,z_1+x_{11};\ldots;\omega,z_n+x_{nm_n}\right)$.
\end{proof}
%
\subsection{Genus $g$ Virasoro graphs for Virasoro $n$-point functions}
The recursive Ward identity \eqref{eq:VirWard} is not manifestly symmetric in $z_{1},\ldots,z_{n}$.  We now describe a symmetric expression for $\calF_{V}(\bm{\omega,z})$. We exploit the differential equations \eqref{eq:nab_Om}--\eqref{eq:nab_som} and \eqref{eq:nablaZM} to develop a symmetric graph-theoretic description. A similar approach is given in \cite{HT} for $g=0,1$ and in \cite{GT} for the $g=2$ surface formed by sewing two tori.

We follow~\cite{HT} to define an order $n$ Virasoro graph  $g_n$ to be a directed graph with $n$ vertices labelled  $z_1,\ldots,z_n$.  Each $z_i$-vertex has degree $\deg(z_i) = 0$, $1$ or $2$. 
The degree-$1$ vertices can have either unit indegree or
outdegree whereas the degree-$2$ vertices have both unit indegree and outdegree. 
The connected subgraphs of $g_n$
consist of $r$-cycles, with $r \ge 1$ degree-$2$ vertices, and
chains with two degree-$1$ end vertices with all other chain vertices of degree 2. We regard a single disconnected degree-$0$ vertex as a degenerate chain. 
Virasoro graphs are in 1-1 correspondence with the set of partial permutations on $n$ objects, i.e.  injective partial mappings from
$\{ z_1,\ldots,z_n\}$ to itself \cite{HT}. Thus the number of Virasoro graphs of a given order $n$ is given by $\sum_{i=0}^{n}i!\binom{n}{i}^2$ e.g. for $n=1$ we obtain 2 graphs and for $n=2$ we obtain 7 graphs.

We define weights on the components of $g_n$ as follows. For each directed edge we define an edge weight
\begin{align}\label{eq:edge}
\calE(z_{i},z_{j})=
\left\{
  \begin{array}{ll}
    \frac{1}{6}s(z_i) &\mbox{ for } z_i=z_j,\\
    \omega(z_i,z_j) &\mbox{ for } z_i\neq z_j.
  \end{array}\right.
\end{align}
We note $\calE(z_{i},z_{j})=\calE(z_{j},z_{i})$. 
Suppose $g_{n}$ contains $M$ disconnected chains $\calC_m$ with initial vertex $x_{m}$ and final vertex $y_{m}$ where $m=1,\ldots,M$ 
\begin{align*}
\xy
(0,0)*[o]=<0.4pt>+{\cir<3pt>{}}="a"*+!R{x_m\,};
(10,0)*[o]=<0.4pt>+{\cir<0pt>{}}="b"*+!D{}
;
(15,0)*[o]=<6pt>+{\cdots\cdots}="e";
(20,0)*[o]=<0.4pt>+{\cir<0pt>{}}="c"*+!D{};
(30,0)*[o]=<0.4pt>+{\cir<3pt>{}}="d"*+!L{\,y_{m}};
\ar "b";"a"; \ar "c";"d"
\endxy.
\end{align*}
Choose  $\tauK:=\ldots,\tau_{ab},\ldots$ for $(a,b)\in\K$ as local coordinates on $\Mg$ as in Lemma~\ref{lem:nablaMperiod}. 
Define a weight associated with the $M$ chains given by a differential operator on $\Mg$ expressed in terms of $\partial_{ab}:=\partial_{\tau_{ab}}$ for $(a,b)\in\K$ as follows
\begin{align}
\Delta_{M}(\bm{x|y})
:=\sum_{\bm{(a,b)}}
\bm{\nu_{ab}(\bm{x,y})\del_{ab}},
 \label{eq:WAggChain}
\end{align}
where the $\bm{(a,b)}$ sum is over all $(a_{1},b_{1}),\ldots,(a_{M},b_{M})\in\K$ and 
with $\nu_{ab}(x,y):=\nu_{a}(x)\nu_{b}(y)$
\begin{align*}
&\bm{\nu_{ab}(\bm{x,y})}:=\nu_{a_{1}b_{1}}(x_{1},y_{1})\ldots \nu_{a_{M}b_{M}}(x_{M},y_{M}),\quad \bm{\del_{ab}}:=\del_{a_{1}b_{1}}\ldots\del_{a_M b_M}.
\end{align*}
The degenerate chain $\degenchain{z}{L}$ has weight $\Delta_{1}(z|z)=\nabla_{\Mg}(z)$ from Lemma~\ref{lem:nablaMperiod}.

We now define the total weight a Virasoro graph $g_n$ containing $L$ cycles and $M$ chains, with endpoints $x_m$, $y_m$ for $m=1,\ldots,M$, by the following differential operator:
\begin{align}
\calD_{g_n}(\bm{z}) := \left(\frac{c}{2}\right)^L \prod_{\mathrm{edges\;}(i,j)}\calE(z_{i},z_{j})\Delta_{M}(\bm{x|y}),\label{gn_weightofgraph}
\end{align}
where the product ranges over all the edges of $g_n$.  $\calD_{g_n}(\bm{z})$ depends on the central charge $c$, the classical differentials $\omega(z_i,z_j)$, $s(z_i)$ and $\nu_a(z_i)$ and $\del_{ab}$ for $(a,b)\in\K$. Summing over the weights of inequivalent order $n$ Virasoro graphs $g_n$, we define the differential operator 
\begin{align}
	\calD_n(\bm{z}) := \sum_{g_n}\calD_{g_n}(\bm{z}),\quad n\ge 1.
	\label{eq:main}
\end{align}
It is useful to also define $\calD_{0}=1$ and $\calD_{-1}=0$.
$\calD_n(\bm{z}) $ is symmetric under permutations of $\bm{z}$ variables since the set of partial permutations is so symmetric. 
It is instructive to consider the first two examples. 
For $n=1$ there are two Virasoro graphs with weights $W\hspace{-1mm}\left(\xy(0,0)*[o]=<0.4pt>+{\cir<3pt>{}}="a"*+!R{z_{1}\,};\endxy\right)=\Delta_1(z_{1}|z_{1})=\nabla_{\Mg}(z_{1})$ and $W(\xy(0,0)*[o]=<0.4pt>+{\cir<3pt>{}}="a"*+!R{z_{1}\,};\ar@(dr,ur) "a";"a";\endxy)=\frac{c}{12}s(z_{1})$ (where $W(g)$ denotes the weight of the displayed graph $g$). Hence
\begin{align}\label{eq:calD1}
	\calD_1(z_{1})=\nabla_{\Mg}(z_{1})+\frac{c}{12}s(z_{1}).
\end{align}
$\calD_{2}(z_{1},z_{2})$ is the sum of 7 Virasoro graph weights. These satisfy the following identities:
\begin{align*}
&W(\degenchain{z_1}{R}\quad\degenchain{z_2}{L})+W(\edge{z_1}{z_2})+W(\edge{z_2}{z_1})
=\nabla_{\Mg,z_1}^{(2)}(z_2)\nabla_{\Mg}(z_{1}),
\\
&W(\twocycle{z_1}{z_2})=\frac{c}{2}\omega(z_1,z_2)^2
=\frac{c}{12}\nabla_{\Mg,z_1}^{(2)}(z_2)s(z_1)+\frac{c}{2}\omega_2(z_1,z_2),
\\
&W\Big(\sloop{z_1}{R}\;\;\degenchain{z_2}{L}\Big)+W\Big(\degenchain{z_1}{R}\;\;\sloop{z_2}{L}\Big)+W\Big(\sloop{z_1}{R}\sloop{z_2}{L}\Big)
\\
&=\frac{c}{12}s(z_1)\nabla_{\Mg}(z_2)+\frac{c}{12}s(z_2)\nabla_{\Mg}(z_1)+\frac{c^2}{144}s(z_1)s(z_2),
\end{align*}
using Lemma~\ref{thm:NablaDelta} and \eqref{eq:nab_som}. Summing  we find
\begin{align}\label{eq:calD2}
\calD_2(z_1,z_2)=\left(\nabla_{\Mg,z_1}^{(2)}(z_2)+\frac{c}{12}s(z_2)\right)
\calD_{1}(z_1)+\frac{c}{2}\omega_2(z_1,z_2)\calD_{0}.
\end{align}
Notice that the order $2$ Virasoro graphs result from adjoining the $z_{2}$ vertex to the order one graphs $\degenchain{z_1}{L}$ and $\xy(0,0)*[o]=<0.4pt>+{\cir<3pt>{}}="a"*+!R{z_{1}\,};\ar@(dr,ur) "a";"a";\endxy$ in all possible ways. In general, all order $n+1$  graphs are of four Types  determined by how vertex $z_{n+1}$ is adjoined to an order $n$ graph $g_{n}$ as follows:
\begin{enumerate}
	\item [(I)]  Vertex $z_{n+1}$ is of degree zero	disconnected from $g_{n}$:  $\xy(0,0)*[o]=<0.4pt>+{\cir<3pt>{}}="a"*+!L{\,z_{n+1}};\endxy$.
	\item [(II)] Vertex $z_{n+1}$ is of degree two
	disconnected from $g_{n}$:  
	$\xy(0,0)*[o]=<0.4pt>+{\cir<3pt>{}}="a"*+!R{z_{n+1}\,};\ar@(dr,ur) "a";"a";\endxy$.  
	\item [(III)] Vertex $z_{n+1}$ is inserted into a $g_{n}$ edge 
	$\xy(0,0)*[o]=<0.4pt>+{\cir<3pt>{}}="a"*+!R{\ldots z_i\,}; (10,0)*[o]=<0.4pt>+{\cir<3pt>{}}="b"*+!L{\,z_j\ldots }; \ar "b";"a";\endxy$ giving
	\[\xy(0,0)*{\cir<3pt>{}}="a"*+!R{\ldots z_i\,}; (10,0)*{\cir<3pt>{}}="b"*+!D{\,z_{n+1}}; (20,0)*{\cir<3pt>{}}="c"*+!L{\,z_j\ldots }; 
	\ar "b";"a"; \ar "c";"b";\endxy.\]
	\item [(IV)] Vertex $z_{n+1}$ is joined to the end vertices $z_{i},z_{j}$ of a $g_{n}$ chain giving two chains:
	\begin{align*}
		\xy
		(0,0)*[o]=<0.4pt>+{\cir<3pt>{}}="a"*+!R{z_{n+1}};
		(10,0)*[o]=<0.4pt>+{\cir<3pt>{}}="b"*+!D{z_{i}\,};
		(20,0)*[o]=<0pt>+{}="f";
		(25,0)*[o]=<6pt>+{\cdots\cdots}="e";
		(30,0)*[o]=<0.4pt>+{\cir<0pt>{}}="c"*+!D{};
		(40,0)*[o]=<0.4pt>+{\cir<3pt>{}}="d"*+!L{\,z_{j}};
		\ar "b";"a"; \ar "f";"b";\ar "d";"c" ;
		\endxy
		\mbox{ and }
		\xy
		(0,0)*[o]=<0.4pt>+{\cir<3pt>{}}="a"*+!R{z_{i}\,};
		(10,0)*[o]=<0.4pt>+{\cir<0pt>{}}="b"*+!D{}
		;
		(15,0)*[o]=<6pt>+{\cdots\cdots}="e";
		(20,0)*[o]=<0pt>+{}="f";
		(30,0)*[o]=<0.4pt>+{\cir<3pt>{}}="c"*+!D{z_{j}};
		(40,0)*[o]=<0.4pt>+{\cir<3pt>{}}="d"*+!L{\,z_{n+1}};
		\ar "b";"a"; \ar "d";"c";\ar "f";"c"
		\endxy.
	\end{align*}
\end{enumerate}
Eqns. \eqref{eq:calD1}and \eqref{eq:calD1} are examples of an important iterative formula.
\begin{theorem}\label{theor:DnIter}
$\calD_{n}(\bm{z})$ obeys the following recursive identity for  all $n\ge 0$:
	\begin{align}\label{eq:Dnrec}
		\notag
		\calD_{n+1}(\bm{z},z_{n+1})&= \left(
		\nabla_{\Mg,\bm{z}}^{(\bm{2})}(z_{n+1})
		+\frac{c}{12}s(z_{n+1})\right)\calD_{n}(\bm{z})
		\\
		&+
		\frac{c}{2}\sum_{k=1}^{n}\omega_2(z_{k},z_{n+1})\calD_{n-1}(\ldots,\widehat{z_k},\ldots),
	\end{align}
for $n$-tuple $\bm{2}=2,\ldots,2$ and where the caret denotes omission of the $z_{k}$ entry and $\omega_{2}(x,y)$ is the symmetric meromorphic $(2,2)$-form of \eqref{eq:omegaN}.
\end{theorem}
To prove Theorem~\ref{theor:DnIter} we need the following lemma concerning derivatives of $\Delta_{M}$:
\begin{lemma}\label{thm:NablaDelta}
	\begin{align*}
		\nabla_{\Mg,\bm{x,y}}^{(\bm{1})}(z)\Delta_M(\bm{x}|\bm{y})
		&=\Delta_{M+1}(\bm{x},z|\bm{y},z) +\sum_{m=1}^{M}\calE(z,x_m)\Delta_M(\ldots,z,\ldots|\bm{y})
		\\
		&
		\quad+\sum_{m=1}^{M}\calE(y_m,z)\Delta_M(\bm{x}|\ldots,z,\ldots)
		,
	\end{align*}
	for $2M$-tuple\footnote{
		By Remark~\ref{rem:delyx}, there is no inconsistency for any degenerate chain with $x_{m}=y_{m}$.}	
	$\bm{1}=1,\ldots,1$ and where in the summands, the vertex label $x_{m}$ (respectively $y_{m}$) is replaced by $z$ in $\Delta_M(\bm{x}|\bm{y})$.
\end{lemma}
\begin{proof}
	From Lemma~\ref {lem:nablaMperiod} and \eqref{eq:nab_nu} and on applying the Leibniz rule of Remark~\ref{rem:delyx} we find
	\begin{align*}
		\nabla_{\Mg,\bm{x},\bm{y}}^{(\bm{1})}(z)\Delta_M(\bm{x}|\bm{y})
		&=
		\sum_{\bm{(a,b)}}\sum_{(a_{M+1},b_{M+1})}
		\nu_{a_{1}b_{1}}(x_1,y_1)\ldots 
		\nu_{a_{M+1}b_{M+1}}(x_{M+1},y_{M+1})
		\bm{\del_{ab}}\del_{a_{M+1} b_{M+1}}
		\\
		&\quad +\sum_{\bm{(a,b)}}
		\nabla_{\Mg,\bm{x},\bm{y}}^{(\bm{1})}(z)
		\left(\nu_{a_{1}b_{1}}(x_1,y_1)\ldots \nu_{a_{M}b_{M}}(x_{M},y_{M})\right)\bm{\del_{ab}}
		\\
		&=\Delta_{M+1}(\bm{x},z|\bm{y},z)
		+\sum_{m=1}^{M}\sum_{\bm{(a,b)}}
		\left(\ldots\omega(z,x_m)\nu_{a_m b_{m}}(z,y_{m})\ldots\right)\bm{\del_{ab}}
		\\
		&\quad+\sum_{m=1}^{M}\sum_{\bm{(a,b)}}
		\left(\ldots\omega(y_m,z)\nu_{a_{m} b_m}(x_{m},z)\ldots\right)\bm{\del_{ab}}
		,
	\end{align*}
	which is equivalent to the stated result on using \eqref{eq:edge}.
\end{proof}
\begin{proof}[Proof of Theorem~\ref{theor:DnIter}]
We first observe that the $\nabla_{\Mg,\bm{z}}^{(\bm{2})}(z_{n+1})\calD_{n}(\bm{z})$ term in
\eqref{eq:Dnrec} contains a 
$\sum_{(a,b)\in\K}\nu_{a}(z_{n+1})\nu_{b}(z_{n+1})\calD_{n}(\bm{z})\partial_{ab}$
contribution which is the sum all Type~(I) graph weights.
Furthermore, the  $\frac{c}{12}s(z_{n+1})\calD_{n}(\bm{z})$ term is the sum all Type~(II) graph weights.
The remaining Type~(III) and (IV) graph weights in \eqref{eq:Dnrec} arise from the action of 
$\nabla_{\Mg,\bm{z}}^{(\bm{2})}(z_{n+1})$ on the $\bm{z}$ dependent parts of $\calD_{n}(\bm{z})$. 
Consider $g_{n+1}$ of Type~(III) where $g_{n}$ contains a 1-cycle  $\xy(0,0)*[o]=<0.4pt>+{\cir<3pt>{}}="a"*+!R{z_{k}\,};\ar@(dr,ur) "a";"a";\endxy$  together with some disconnected graph $g_{n-1}$ so that $W(g_{n})=\frac{c}{12}s(z_{k})W(g_{n-1})$. Using~\eqref{eq:nab_som} we find the RHS of \eqref{eq:Dnrec} gives rise to a contribution
\begin{align*}
	W(g_{n-1})\frac{c}{12}\nabla_{\Mg,z_{k}}^{(2)}(z_{n+1})s(z_{k})=\frac{c}{2}\omega(z_{k},z_{n+1})^{2}W(g_{n-1})-\frac{c}{2}\omega_{2}(z_{k},z_{n+1})W(g_{n-1}).
\end{align*}
The $-\frac{c}{2}\omega_{2}(z_{n+1},z_{k})W(g_{n-1})$ term is canceled by a summand term in  \eqref{eq:Dnrec}. Thus we obtain the weights of all Type~(III)  graphs of the form $\xy 
(0,0)*[o]=<0.4pt>+{\cir<3pt>{}}="a"*+!R{z_{n+1}\,}; (10,0)*[o]=<0.4pt>+{\cir<3pt>{}}="b"*+!L{\,z_{k}}; \ar@/^/ "b";"a";\ar@/^/ "a";"b";\endxy$ with disconnected  $g_{n-1}$.
The remaining Type~(III) terms come from any $g_{n}$ edge with vertices  $z_i\neq z_j$. We find such a contribution arises from  terms in the RHS of \eqref{eq:Dnrec} of the form
\begin{align*}
\nabla_{\Mg,z_i,z_j}^{(1,1)}(z_{n+1})\omega(z_i,z_j)=\omega(z_i,z_{n+1})\omega(z_{n+1},z_j)=
W_{\calE}(\xy(0,0)*{\cir<3pt>{}}="a"*+!R{z_i\,}; (10,0)*{\cir<3pt>{}}="b"*+!D{\,z_{n+1}}; (20,0)*{\cir<3pt>{}}="c"*+!L{\,z_j}; \ar "b";"a"; \ar "c";"b";\endxy),
\end{align*}
using \eqref{eq:nab_ombidiff} and where $W_{\calE}$ denotes the edge weights of the displayed graph.

Finally, we consider differential operators arising from the weights of disconnected chains in $g_{n}$. Thus if $g_{n}$ contains a single degenerate chain $\degenchain{z_i}{L}$, 
 Lemma~\ref{thm:NablaDelta} implies
\begin{align*}
&\nabla_{z_i}^{(2)}(z_{n+1})W(\,\degenchain{z_i}{L})
=\nabla_{z_i}^{(2)}
(z_{n+1})\Delta_{1}(z_i|z_i)
\\
&=\Delta_2(z_i,z_{n+1}|z_i,z_{n+1})
+
\calE(z_{n+1},z_i)\Delta(z_{n+1}|z_i)
+
\calE(z_i,z_{n+1})\Delta(z_i|z_{n+1})
\\
&=
W(\xy(0,0)*[o]=<0.4pt>+{\cir<3pt>{}}
="a"*+!R{z_i\,}; (10,0)*[o]=<0.4pt>+
{\cir<3pt>{}}="b"*+!L{\,z_{n+1}};\endxy)
+W(\edge{z_{n+1}}{z_i})
+W(\edge{z_i}{z_{n+1}})
,
\end{align*}
corresponding to adjoining $z_{n+1}$ with $\degenchain{z_i}{L}$.
In general, if $g_{n}$ contains $M$ chains with end vertices $x_{m},y_{m}$ for $m=1,\ldots ,M$, then  Lemma~\ref{thm:NablaDelta} implies
\begin{align*}
&\nabla_{\Mg,\bm{x},\bm{y}}^{(\bm{k})}(z_{n+1})
W_{\calC}\left(\substack{
\xy
(0,0)*[o]=<0.4pt>+{\cir<3pt>{}}="a"*+!R{x_1\,};
(10,0)*[o]=<0.4pt>+{\cir<0pt>{}}="b"*+!D{};
(15,0)*[o]=<6pt>+{\cdots\cdots}="e";
(20,0)*[o]=<0.4pt>+{\cir<0pt>{}}="c"*+!D{};
(30,0)*[o]=<0.4pt>+{\cir<3pt>{}}="d"*+!L{\,y_{1}};
\ar "b";"a"; \ar "c";"d"
\endxy
\\
\vdots\\
\xy
(0,0)*[o]=<0.4pt>+{\cir<3pt>{}}="a"*+!R{x_{M}\,};
(10,0)*[o]=<0.4pt>+{\cir<0pt>{}}="b"*+!D{};
(15,0)*[o]=<6pt>+{\cdots\cdots}="e";
(20,0)*[o]=<0.4pt>+{\cir<0pt>{}}="c"*+!D{};
(30,0)*[o]=<0.4pt>+{\cir<3pt>{}}="d"*+!L{\,y_{M}};
\ar "b";"a"; \ar "c";"d"
\endxy}\right)
=W_{\calC}\left(\substack{
\xy
(0,0)*[o]=<0.4pt>+{\cir<3pt>{}}="a"*+!R{x_1\,};
(10,0)*[o]=<0.4pt>+{\cir<0pt>{}}="b"*+!D{};
(15,0)*[o]=<6pt>+{\cdots\cdots}="e";
(20,0)*[o]=<0.4pt>+{\cir<0pt>{}}="c"*+!D{};
(30,0)*[o]=<0.4pt>+{\cir<3pt>{}}="d"*+!L{\,y_{1}};
\ar "b";"a"; \ar "c";"d"
\endxy
\\
\vdots\\
\xy
(0,0)*[o]=<0.4pt>+{\cir<3pt>{}}="a"*+!R{x_{M}\,};
(10,0)*[o]=<0.4pt>+{\cir<0pt>{}}="b"*+!D{};
(15,0)*[o]=<6pt>+{\cdots\cdots}="e";
(20,0)*[o]=<0.4pt>+{\cir<0pt>{}}="c"*+!D{};
(30,0)*[o]=<0.4pt>+{\cir<3pt>{}}="d"*+!L{\,y_{M}};
\ar "b";"a"; \ar "c";"d"
\endxy\\
\degenchain{z_{n+1}}{L}}\right)
\\
&\quad\quad +\sum_{m=1}^{M}W_{\calE}(\edge{y_m}{z_{n+1}})
W_{\calC}(
\xy
(0,0)*[o]=<0.4pt>+{\cir<3pt>{}}="a"*+!R{x_{m}\,};
(10,0)*[o]=<0.4pt>+{\cir<0pt>{}}="b"*+!D{};
(15,0)*[o]=<6pt>+{\cdots\cdots}="e";
(20,0)*[o]=<0.4pt>+{\cir<0pt>{}}="c"*+!D{};
(22.5,0)*[o]=<0.4pt>+{\cir<3pt>{}}="c"*+!D{y_m};
(32.5,0)*[o]=<0.4pt>+{\cir<3pt>{}}="d"*+!L{\,z_{n+1}};
\ar "b";"a"; \ar "c";"d"
\endxy)
\\
&
\quad\quad+\sum_{m=1}^{M}
W_{\calE}(\edge{z_{n+1}}{x_m})
W_{\calC}(\xy
(0,0)*[o]=<0.4pt>+{\cir<3pt>{}}="d"*+!R{\,z_{n+1}};
(10,0)*[o]=<0.4pt>+{\cir<3pt>{}}="a"*+!D{x_{m}\,};
(12.5,0)*[o]=<0.4pt>+{\cir<0pt>{}}="b"*+!D{};
(17.5,0)*[o]=<6pt>+{\cdots\cdots}="e";
(22.5,0)*[o]=<0.4pt>+{\cir<0pt>{}}="c"*+!D{};
(32.5,0)*[o]=<0.4pt>+{\cir<3pt>{}}="f"*+!L{\,y_m};
\ar "a";"d"; \ar "c";"f"
\endxy)
\end{align*}
where $W_{\calC}$ denotes the chain weight contribution.
Note that the first term is of Type~(I) as already discussed whereas the additional $W_{\calE}$ terms are all of Type~(IV) found by adjoining $z_{n+1}$ to chain endpoint  vertices  $x_{m}$ or $y_{m}$. 
Thus, altogether, we confirm that \eqref{eq:Dnrec} holds since all order $n+1$ graph weights occur in the RHS of \eqref{eq:Dnrec}.
\end{proof}

We define normalised partition and Virasoro generating $n$-point  functions by
\begin{align*}\Theta_V:= Z_V Z_{M}^{-c},\quad 
	G_{n}(\bm{z}):=\calF_V(\bm{\omega,z})Z_{M}^{-c},
\end{align*}
where $Z_{M}$ is the rank one genus $g$ Heisenberg VOA partition function and $c$ is the central charge for the given VOA $V$. 
The normalising factor will allow us to exploit the differential equations~\eqref{eq:nab_Om}--\eqref{eq:nab_som} and the Ward identity Proposition~\ref{prop:Ward} in Theorem~\ref{theor:DnIter} below.
From Proposition~\ref{prop:ZVqexp} we know that the expansion of $\Theta_{V}$ to any finite order in $q_{1},\ldots,q_{g}$ is convergent on $\Schg$. 
Thus we may use  Lemmas~\ref{lem:nablaM} and \ref{lem:nablayM} to replace the action on $\Theta_{V}$ of $\nabla_{\bm{y}}^{(\bm{m})}(x)$ by 
$\nabla_{\Mg,\bm{y}}^{(\bm{m})}(x)$ throughout. We again choose  $\tauK:=\ldots,\tau_{ab},\ldots$ for $(a,b)\in\K$ as local coordinates on $\Mg$.
\begin{theorem}\label{theor:main2}
	The order $n$ genus $g$ Virasoro generating function is given by \begin{align}\label{eq:GnEqn}
		G_n(\bm{z})=\calD_n(\bm{z})\Theta_{V}.
	\end{align}
\end{theorem}
\begin{proof}
We prove the result by induction in $n$. Eqn.~\eqref{eq:GnEqn} is trivially true for $n=0$ since $\calD_{0}=1$. For $n=1$ we know from \eqref{eq:Fom} and \eqref{eq:nablaZM} that
\[
G_{1}(z_{1})=Z_{M}^{-c}\,\nabla_{\Mg}(z_{1})\left(Z_{M}^{c}\Theta_{V}\right)=\calD_{1}(z_{1})\Theta_{V},
\]
recalling \eqref{eq:calD1}.
The Virasoro Ward identity \eqref{eq:VirWard} implies 
\begin{align*}
G_{n+1}(\bm{z},z_{n+1})
	&=Z_{M}^{-c}\,\nabla_{\Mg,\bm{z}}^{(\bm{2})}(z_{n+1})\left(G_{n}(\bm{z})Z_{M}^{c}\right)
	+\frac{c}{2}\sum_{k=1}^{n}\omega_{2}(z_{n+1},z_{k})G_{n-1}(\ldots;\widehat{z_{k}};\ldots)
	\\
	&=\left(
	\nabla_{\Mg,\bm{z}}^{(\bm{2})}(z_{n+1})
	+\frac{c}{12}s(z_{n+1})\right)G_{n}(\bm{z})
	+\frac{c}{2}\sum_{k=1}^{n}\omega_{2}(z_{k},z_{n+1})G_{n-1}(\ldots;\widehat{z_{k}};\ldots),
\end{align*}
using \eqref{eq:nablaZM} again. By induction, we assume that $G_n(\bm{z})=\calD_n(\bm{z})\Theta_{V}$ and $G_{n-1}(\bm{z})=\calD_{n-1}(\bm{z})\Theta_{V}$.
Then we find $G_{n+1}(\bm{z},z_{n+1})=\calD_{n+1}(\bm{z},z_{n+1})\Theta_{V}$ using Theorem~\ref{theor:DnIter}.
\end{proof}
\begin{remark}\label{rem:V modules}
	In \cite{TW1} we discuss genus $g$ partition and correlation functions for given  $V$-modules
	(subject to some mild fusion rule assumptions)  where the $g$ $V$-basis sums of \eqref{eq:Zg} are generalised to $V$-module basis sums together with suitable intertwiner vertex operator insertions. In particular, there is a natural generalisation of the Ward identity Proposition~\ref{prop:Ward} and the Virasoro generating function formula of Theorem~\ref{theor:DnIter} where $\Theta_{V}$ is replaced by the normalised partition function for the given $V$-modules.
\end{remark}

\section{Modular transformations}\label{sect:an_mod}
\subsection{Modular transformations of classical differentials}
Consider the change of homology basis from the original marking $(\bm{\alpha},{\bm{\beta}})$ to  
$(\widetilde{\bm{\alpha}},\widetilde{\bm{\beta}})$  where e.g. \cite{Mu1,Fa}
\begin{align}
	\begin{bmatrix}
		\widetilde{\bm{\beta}}\\
		\widetilde{\bm{\alpha}}
	\end{bmatrix}
	=
	\begin{bmatrix}
		A & B\\
		C & D 
	\end{bmatrix}
	\begin{bmatrix}
		{\bm{\beta}}\\
		{\bm{\alpha}}
	\end{bmatrix},
	\label{eq:hommod}
\end{align}
for a symplectic  modular matrix $\left[
\begin{smallmatrix}
	A & B\\
	C & D 
\end{smallmatrix}
\right]\in\Sp(2g,\Z)$ i.e. 
\begin{align}\label{eq:modinv}
	\begin{bmatrix}
		A & B\\
		C & D 
	\end{bmatrix}^{-1} =
	\begin{bmatrix}
		D^{T} & -B^{T}\\
		-C^{T} & A^{T}
	\end{bmatrix},
\end{align}
and $A,B,C,D$ obey the relations:
\begin{align}
	&AD^{T}-BC^{T}=A^{T}D-C^{T}B=\Ig,
	\label{eq:AD-BC}
	\\
	&AB^{T},\, A^{T}C,\,DC^{T},\, D^{T}B \mbox{ are symmetric},
	\label{eq:ABT_sym}
\end{align}
for identity matrix $\Ig$. 

Let $\nu (x):=\left[\nu_{1}(x),\ldots ,\nu_{g}(x)\right] $ be the row vector of $1$-forms \eqref{eq:nu}.  It is convenient to define 
\begin{align}\label{eq:calMNdef}
	\calM :=C\Omega+D,\quad \calN :=\calM ^{-1}.
\end{align}
Using $\oint\limits_{\alpha_{a}}\nu_{b}=\tpi\, \delta_{ab}$ then \eqref{eq:hommod} implies  
$\nu (x)$ transforms as
\begin{align}
	\widetilde{\nu }(x)=\nu(x)\calN,
	\label{eq:numod}
\end{align}
This implies the period matrix \eqref{eq:period} transforms as
\begin{align}
	\Omt=(A\Omega +B)\calN .
	\label{eq:Omegamod}
\end{align}
\begin{lemma}\label{lem:oms_mod}
	$\omega(x,y)$ and $s(x)$ transform under $\Sp(2g,\Z)$ as follows:
	\begin{align}
		\widetilde{\omega}(x,y)&=\omega(x,y)
		-\frac{1}{\tpi}\nu (x)\calN C\nu ^T(y),
		\label{eq:ommod}
		\\
		\widetilde{s}(x)&=s(x)
		-\frac{3}{\pi \im}\nu (x)\calN C\nu ^T(x),
		\label{eq:smod}
	\end{align}
where $\calN C$ is symmetric.
\end{lemma}
\begin{proof}
	$f(x,y):=\widetilde{\omega}(x,y)-\omega(x,y)$ is a holomorphic symmetric  $(1,1)$ form in $(x,y)$ from \eqref{eq:omega}. Thus $f(x,y) =\nu (x)\calX \nu ^T(y)$  for a symmetric matrix $\calX$ determined by
	\[	\frac{1}{\tpi} \oint_{\alpha_{a}}\widetilde{\omega}(x,\cdot)
	=\frac{1}{\tpi} \oint_{\alpha_{a}}f(x,\cdot)=(\nu(x) \calX)_{a}.
	\]
	But \eqref{eq:modinv} implies $\alpha_{a}=\sum_{c\in\Ip}\left(-\betat_{c}C_{ca}+\alphat_{c}A_{ca}\right)$ so that 
	\[  
    \oint_{\alpha_{a}}\widetilde{\omega}(x,\cdot)=-\sum_{c\in\Ip}\nut_{c}(x)C_{ca}+0=-\left(\nu(x) \calN C\right)_{a},
	\]
	using \eqref{eq:nu}.
	Thus $\calX=-\frac{1}{\tpi} \calN C$ giving \eqref{eq:ommod}. Eqn.~\eqref{eq:smod} follows using \eqref{eq:projcon}.
\end{proof}
$\calN$ enjoys the following properties:
\begin{lemma}\label{lem:NOmt} \leavevmode
	\begin{enumerate}
		\item [(i)] $\calN=A^{T}-C^{T}\Omt$.
		\item [(ii)] $\calN C$ is given by
		\begin{align*}
			(\calN C)_{ab}&=
			\begin{cases}
				\partial_{\Omega_{aa}}\log\det \calM   \mbox{ for }  a=b
				\\
				\half \partial_{\Omega_{ab}} \log\det \calM   \mbox{ for }  a\neq b.
			\end{cases}
		\end{align*}
	\end{enumerate}
\end{lemma}
\begin{proof}
	$A^{T}-C^{T}\Omt=(A^{T}(C\Omega+D)-C^{T}(A\Omega+B))\calN =\calN $ using \eqref{eq:AD-BC} and \eqref{eq:ABT_sym}.  Thus, as already found, $\calN C=A^{T}C-C^{T}\Omt \,C$ is symmetric from \eqref{eq:ABT_sym}.
	The dependence of $\det \calM $ on $\Omega_{aa}$ arises in the $a$-th column of $\calM $ so that
	\begin{align*}
		\partial_{\Omega_{aa}}\det \calM  &
		=\sum_{b\in\Ip}(C_{ba}+0)\cof(\calM )_{ba}=(\adj(\calM )C)_{aa},
	\end{align*}
	for adjugate  $\adj(\calM )$. Thus  $\partial_{\Omega_{aa}}\log \det \calM = (\calN C)_{aa}$ since $\calN :=\calM ^{-1}=\adj(\calM )/\det \calM $. A similar analysis applied for the $a\neq b$ cases where the dependence of $\det \calM $ in $\Omega_{ab}=\Omega_{ba}$  arises both in the $a$-th and $b$-th columns  of $\calM $.
\end{proof}
Combining Lemmas~\ref{lem:oms_mod} and  \ref{lem:NOmt} and recalling Lemma~\ref{lem:nablaMperiod} we obtain \cite{Fa}:
\begin{corollary} $\omega(x,y)$ and $s(x)$ transform under $\Sp(2g,\Z)$ as follows
	\begin{align}
		\widetilde{\omega}(x,y)&=\omega(x,y)
		-\frac{1}{2}
		\sum_{1\le a\le b\le g}\left(\nu_{a}(x)\nu_{b}(y)+\nu_{b}(x)\nu_{a}(y)\right)
		\partial_{ab} \log\det \calM 
		\label{eq:ommod2}
		\\
		\widetilde{s}(x)&=s(x)
		-6\sum_{1\le a\le b\le g}
		\nu_{a}(x)\nu_{b}(x)
		\partial_{ab} \log\det \calM 
		\notag
		\\
		&= s(x)
		-6\, \nabla_{\Mg}(x) \log\det \calM ,
		\label{eq:smod3}
	\end{align}
	where  $\partial_{ab}:=\partial_{\tau_{ab}}=\frac{1}{\tpi}\partial_{\Omega_{ab}}$. 
\end{corollary}
We now consider the modular properties  of the differential operators $	\nabmy{\bfm}{\Mg,\bfy}(x)$ and $\calD_{n}$ that appear in \eqref{eq:nab_Om}--\eqref{eq:nab_som} and \eqref{eq:GnEqn} respectively. 
The change of homology basis \eqref{eq:hommod} determines a new marking $(\widetilde{\bm{\alpha}},\widetilde{\bm{\beta}})$ on an isomorphic Riemann surface for which there exists a Schottky uniformisation for some choice of parameters $\widetilde{w}_{a}, \widetilde{w}_{-a}$ and $\widetilde{\rho}_{a}$, for $a\in\Ip$. 
We let $\PsiMt(x,y)$ denote the corresponding weight $(2,-1)$ quasiform of Lemma~\ref{lem:nablayM} and $\nablat_{\Mg}(x)$ and $\nabmyt{\bfm}{\Mg,\bfy}(x)$ the differential operators corresponding to Lemma~\ref{lem:nablaM} and \eqref{eq:nablayMg}, respectively.
\begin{proposition}\label{prop:nabla_mod}
	Under $\Sp(2,\Z)$ modular transformations of the homology basis we find
	\begin{enumerate}
		\item [(i)]$\nablat_{\Mg}(x) =	\nabla_{\Mg}(x)$,
		\item [(ii)]$\PsiMt(x,y)=\PsiM(x,y)$.
		\item [(iii)] $\nabmyt{\bfm}{\Mg,\bfy}(x) =	\nabmy{\bfm}{\Mg,\bfy}(x)$.
	\end{enumerate}
\end{proposition}
\begin{proof}
We first prove (i). Rauch's formula \eqref{eq:nab_Om} in the $(\widetilde{\bm{\alpha}},\widetilde{\bm{\beta}})$ marking gives
	\begin{align}\label{eq:Rauchmod}
		\nablat_{\Mg}(x)\widetilde{\tau}_{ab}
		=\widetilde{\nu}_{a}(x)\widetilde{\nu}_{b}(x).
	\end{align}
	Consider 
	\begin{align*}
		\sum_{b\in\Ip}\nabla_{\Mg}(x)\left(\Omt_{ab}\right)\calM_{bc}&=
		\sum_{b\in\Ip}\nabla_{\Mg}(x)\left((A\Omega+B)\calN \right)_{ab}\calM_{bc}
		\\
		&=
		\nabla_{\Mg}(x)(A\Omega+B)_{ac}
		-\sum_{b\in\Ip}\Omt_{ab}\nabla_{\Mg}(x)\calM_{bc},
	\end{align*}
	using $\calN \calM =\Ig$. Rauch's formula in the original  $(\bm{\alpha,\beta})$ marking implies
	\begin{align*}
		\tpi\sum_{b\in\Ip}\nabla_{\Mg}(x)\left(\Omt_{ab}\right)\calM_{bc}=
		(\nu(x)(A^{T} -C^{T}\Omt) )_{a} \nu_{c}(x)=\nut_{a}(x)\nu_{c}(x),
	\end{align*}
using Lemma~\ref{lem:NOmt} and \eqref{eq:numod}. 
	Multiplying by $\calN $ we therefore find $\nabla_{\Mg}(x)\widetilde{\tau}_{ab}=\nut_{a}(x)\nut_{b}(x)$. 
	Thus comparing to \eqref{eq:Rauchmod} we find $\nabla_{\Mg}(x)\widetilde{\tau}_{ab}=\nablat_{\Mg}(x)\widetilde{\tau}_{ab}$. 
	Choosing any $3g-3$ locally independent components of $\widetilde{\tau}$ it follows that $ \nabla_{\Mg}(x)=\nablat_{\Mg}(x)$. 
	
	We next prove (ii). In  the $(\alphat,\betat)$ marking,  \eqref{eq:nab_nu} gives
	\begin{align}\label{eq:nabnumod}
		\nabla_{\Mg}(x) \,\nut_{a}(y) +
		d_{y}\left(\PsiMt(x,y)\,\nut_{a}(y)\right)=\omegat(x,y)\nut_{a}(x),
	\end{align}
	using part (i). Consider 
	\begin{align*}
		\sum_{a\in\Ip}\nabla_{\Mg,y}^{(1)}(x)\left(\nut_{a}(y)\right)\calM _{ab} &=
		\sum_{a\in\Ip}\nabla_{\Mg}(x) \left(\nut_{a}(y)\right) \calM _{ab}+
		d_{y}\left(\PsiM(x,y)\,\nu_{b}(y)\right)
		\\
		&= 
		\nabla_{\Mg}(x) \nu_{b}(y)-\sum_{a\in\Ip}(\nu(y) \calN )_{a}\nabla_{\Mg}(x) \calM_{ab}+
		d_{y}\left(\PsiM(x,y)\,\nu_{b}(y)\right)
		\\
		&=\nabla_{\Mg,y}^{(1)}(x)\,\nu_{b}(y)-\frac{1}{\tpi}\nu(y)\calN C\nu^{T}(x) \nu_{b}(x),
	\end{align*}
	much as before. Multiplying by $\calN$ and using \eqref{eq:nab_nu} we obtain 
	\begin{align*}
		\nabla_{\Mg,y}^{(1)}(x) \nut_{a}(y)=
		\left(\omega(x,y)-\frac{1}{\tpi}\nu(y)\calN C\nu^{T}(x) \right)\nut_{a}(x)
		=\omegat(x,y)\nut_{a}(x),
	\end{align*}
	from \eqref{eq:numod}. Comparing with \eqref{eq:nabnumod} we conclude that for all $a\in\Ip$
	\begin{align*}
		d_{y}\left(\left(\PsiMt(x,y)-\PsiM(x,y)\right)\,\nut_{a}(y)\right)=0.
	\end{align*}
$\PsiMt(x,y)-\PsiM(x,y)$ is holomorphic in $x$ and $y$ so that  $\PsiMt(x,y)-\PsiM(x,y)=\Theta(x)\nut_{a}(y)^{-1}$ for some $\Theta\in\calH_{2}$. But $\nut_{a}(y)^{-1}$ cannot be holomorphic in $y$ by the Riemann-Roch theorem\footnote{The space of holomorphic weight $-1$ forms is trivial for $g\ge 2$.}. Therefore (ii) holds. (i) and (ii) imply (iii) from definition \eqref{eq:nablayMg}.
\end{proof}
\begin{remark}
One can also prove $\PsiMt(x,y)=\PsiM(x,y)$ directly from the expression \eqref{eq:PsiM} together with \eqref{eq:numod}, \eqref{eq:ommod} and that $\nabla_{\Mg}(x)\calN=-\calN\left(\nabla_{\Mg}(x)\calM\right)\calN=-\calN C\nu(x)^{T }\nu(x)\calN$.
\end{remark}
Lastly, we consider the modular properties of  $\calD_n(\bm{z})$. $\calD_n(\bm{z})$ depends on $c$, $\omega(x,y)$, $\nu_{a}(x)$,  $s(x)$ and $\partial_{ab}$  
for $3g-3$ locally  independent $\tau$ components $\tauK:=\ldots,\tau_{ab},\ldots$ for $(a,b)\in\K$.
Let $\widetilde{\calD}_{n}(\bm{z})$ denote the $\Sp(2g,\Z)$ transformation \eqref{eq:hommod} of $\calD_n(\bm{z})$ using \eqref{eq:numod}-\eqref{eq:smod}.
\begin{theorem}
	\label{theor:Dnmod}
Let $F(\tauK)$ be a locally differentiable function on $\Mg$. 
	\begin{align}
		\widetilde{\calD}_{n}(\bm{z})
		 \left( \det (\calM)^{c/2}F(\tauK)\right)=
		\det(\calM )^{c/2}\calD_n(\bm{z})  F(\tauK).
		\label{eq:OnMod}
	\end{align}
\end{theorem}
\begin{proof}
	We prove the result by induction in $n$. The result is trivially true for $n=0$. For $n=1$  we have 
	\begin{align*}
		\widetilde{\calD}_{1}(z_{1}) \left( \det (\calM )^{c/2}F\right)&=
		\left(\nabla_{\Mg}(z_{1})+ \frac{c}{12}\widetilde{s}(z_{1})\right)\left( \det (\calM )^{c/2}F\right)\\
		&=\det (\calM )^{c/2}\calD_{1}(z_{1})F,
	\end{align*}
	using  Proposition~\ref{prop:nabla_mod}~(i) and \eqref{eq:smod3}.
	By induction let us assume that \eqref{eq:OnMod} holds for $n$ and $n-1$.
Then Theorem~\ref{theor:DnIter} implies that
	\begin{align*}
		\widetilde{\calD}_{n+1}(\bm{z},z_{n+1})\det (\calM )^{c/2}F
		&=
		 \left(
		 \nabla_{\Mg,\bm{z}}^{(\bm{2})}(z_{n+1})
		 +\frac{c}{12}\widetilde{s}(z_{n+1})\right)
		 \left( \det (\calM )^{c/2}{\calD}_{n}(\bm{z})F\right)
		 \\
		 &
		\quad+
		\frac{c}{2}\sum_{k=1}^{n}\omega_2(z_{k},z_{n+1})\det (\calM )^{c/2}
		\calD_{n-1}(\ldots,\widehat{z_k},\ldots)F,
	\end{align*}
where from Proposition~\ref{prop:nabla_mod} we know
$\nablat_{\Mg,\bm{z}}^{(\bm{2})}(z_{n+1})=\nabla_{\Mg,\bm{z}}^{(\bm{2})}(z_{n+1})$ and $\omegat_{2}(x,y)=\omega_{2}(x,y)$ (since $\PsiMt(x,y)=\PsiM(x,y)$). 
Then \eqref{eq:smod3} implies 
\begin{align*}
&\widetilde{\calD}_{n+1}(\bm{z},z_{n+1})\det (\calM )^{c/2}F
\\
&=
\det (\calM )^{c/2} \left(\left(
\nabla_{\Mg,\bm{z}}^{(\bm{2})}(z_{n+1})
+\frac{c}{12}s(z_{n+1})\right)
 {\calD}_{n}
+\frac{c}{2}\sum_{k=2}^n \omega_{2}(z_{1},z_k) {\calD}_{n-1}\right)F
\\
&=\det (\calM )^{c/2}{\calD}_{n+1}(\bm{z},z_{n+1})F,
\end{align*}
by Theorem~\ref{theor:DnIter} again.
\end{proof}
\begin{corollary}\label{cor:Teich}
	Suppose that $F_{k}(\tauK)$ is a weight $k$ Teichm\"uller form on $\Mg$ for $\Sp(2g,\Z)$ with a multiplier system. Then, with $c=2k$, $\calD_{n}(\bm{z})F_{k}(\tauK)$ also transforms like a weight $k$ Teichm\"uller form with the same multiplier system.
\end{corollary}
\begin{proof}
Let $\widetilde{\tau}_{\calK}$ denote an $\Sp(2g,\Z)$ transformation on ${\tauK}$. Then $F_{k}(\widetilde{\tau}_{\calK})=\varepsilon\det(\calM)^{k}F_{k}(\tauK)$ for some multiplier $\varepsilon$. Therefore \eqref{eq:OnMod} implies that 
\begin{align*}
	\widetilde{\calD}_{n}(\bm{z})F_{k}(\widetilde{\tau}_{\calK})
	=\varepsilon \widetilde{\calD}_{n}(\bm{z})
	\left( \det (\calM)^{k}F_{k}(\tauK)\right)=
	\varepsilon\det(\calM )^{k}\calD_n(\bm{z})  F_{k}(\tauK).
\end{align*}
\end{proof}
We conclude with some brief remarks on the significance of  Theorems~\ref{theor:DnIter} and \ref{theor:Dnmod} in forthcoming work \cite{T2}. Following the seminal work of Friedan and Shenker in \cite{FS}, it is believed that the conformal anomaly in CFT obstructs the existence of a global partition function on moduli space for $g\ge 2$. The genus $g$ partition $Z_{M}^2$ function \cite{T1} for the rank 2 Heisenberg VOA $M^{2}$  is given by the Montonen-Zograf product formula \cite{Mo,Zo} which is 
the holomorphic part of the Hodge line bundle on $\Schg$.  $Z_{M}^2$ is a convergent function on  $\Schg$ but by Mumford's theorem \cite{Mu2} cannot be projected down to moduli space $\Mg$ (due to a non-zero Chern class which corresponds to the conformal anomaly in physics language). However, according to Friedan and Shenker, all (suitable) CFTs of a given central charge $c$ are believed to share the same conformal anomaly. In the language of this paper, it is conjectured that  $\Theta_{V}=Z_{V}Z_{M}^{-c}$ can be globally defined on $\Mg$ for some suitable class of VOAs.
For a lattice VOA $V_{L}$, for a rank $c$ even Euclidean lattice $L$,  we know \cite{T1} that $\Theta_{V_{L}}$ is the Siegel lattice theta series $\sum_{\bm{\lambda}\in L^{g}}e^{\im\pi\,\bm{\lambda}.\Omega.\bm{\lambda}}$ (where the sum is over $(\lambda_{1},\ldots,\lambda_{g})$ for all $\lambda_{a}\in L$).
$\Theta_{V_{L}}$ is a Teichm\"uller modular form, with a multiplier system, of weight $c/2$ and globally defined on $\Mg$. 
It is also been shown in \cite{C} that for a holomorphic $V$  then $\Theta_{V}$ is a Teichm\"uller form of weight $c/2$. 

For many rational VOAs, such as Virasoro minimal models (e.g. \cite{DFMS}), there exists a singular Virasoro descendant vector $v$ of weight $2k\ge 4$ with zero 1-point genus one correlation function giving rise to a modular differential equation of order $k$ for the genus one partition function for $V$ and its modules. Modular differential equations based on Virasoro vectors can also be constructed by other methods e.g. \cite{T3}. These approaches can be extended to genus $g$ partition functions. In particular, we may extract  $Z_{V}(v,x)Z_{M}^{-c}$ from $\calD_{k}(\bm{z}) \Theta_{V}$ as described in the proof of Proposition~\ref{prop_Ggen}. Then, for a Virasoro singular vector, $Z_{V}(v,x)=0$ determines an order $k$ partial differential equation for $\Theta_{V}$ whose coefficients are holomorphic $2k$-forms. 
Following Remark~\ref{rem:V modules}, we also find that the normalized partition function for any $V$-modules, satisfy the same partial differential equation. Lastly, provided the various partition functions are globally defined on $\Mg$, they form a vector-valued Teichm\"uller form, with a multiplier system, of weight $c/2$ using Corollary~\ref{cor:Teich} to Theorem~\ref{theor:Dnmod}. Examples of VOAs with order 4 partial differential equations for $\Theta_{V}$ (and $V$-module partition functions) will be  discussed in \cite{T2}.

\end{document}